\documentclass[10pt]{article}

\usepackage[utf8]{inputenc}
\usepackage[english]{babel}
\usepackage{amsmath}
\usepackage{amssymb}
\usepackage{amsthm}
\usepackage{mathtools}
\usepackage{thmtools}
\usepackage{url}
\usepackage{tikz}
\usepackage{indentfirst}
\usepackage{layouts}
\usepackage{enumitem}
\usepackage{longtable}
\usepackage{hyperref}
\usepackage[
backend=biber,
style=alphabetic,
sorting=nty,
maxbibnames=99
]{biblatex}

\newtheorem{theorem}{Theorem}
\declaretheorem[style=thmbluebox,name=Theorem,numbered=no]{theorem*}
\declaretheorem[style=thmbluebox,name=Proposition,sibling=theorem]{proposition}
\declaretheorem[style=thmbluebox,name=Lemma,sibling=theorem]{lemma}
\declaretheorem[style=thmbluebox,name=Corollary,sibling=theorem]{corollary}
\declaretheorem[style=thmbluebox,name=Lemma,numbered=no]{lemma*}
\declaretheorem[style=thmbluebox, sibling=theorem]{definition}

\usepackage{todonotes} % This is for writing comments.
\title{Ratios of Naruse-Newton Coefficients Obtained from Descent Polynomials}
\author{Andrew Cai\footnote{Clements High School, Sugar Land, TX. \newline \indent \space \space Author's Email Address: andrewcai31@gmail.com}}
\date{\today}

\definecolor{light-gray}{gray}{0.75}

\begin{document}

\maketitle

\begin{abstract}
We study Naruse-Newton coefficients, which are obtained from expanding descent polynomials in a Newton basis introduced by Jiradilok and McConville. These coefficients $C_0, C_1, \ldots$ form an integer sequence associated to each finite set of positive integers. For fixed nonnegative integers $a<b$, we examine the set $R_{a, b}$ of all ratios $\frac{C_a}{C_b}$ over finite sets of positive integers. We characterize finite sets for which $\frac{C_a}{C_b}$ is minimized and provide a construction to prove $R_{a, b}$ is unbounded above. We use this construction to obtain results on the closure of $R_{a, b}$. We also examine properties of Naruse-Newton coefficients associated with doubleton sets, such as unimodality and log-concavity. Finally, we find an explicit formula for all ratios $\frac{C_a}{C_b}$ of Naruse-Newton coefficients associated with ribbons of staircase shape.
\end{abstract}

\section{Introduction}
For a positive integer $n$, the \emph{descent set} of a permutation $\pi \in S_n$ is defined as the set of positions $i \in [n-1] = \{1, 2, \ldots, n-1\}$ such that $\pi_i>\pi_{i+1}$.  For a fixed set of positive integers $I$, the univariate \emph{descent polynomial} $d_I$ is the unique polynomial such that for all positive integers $n$ with $n > \max(I \cup \{0\})$, the value $d_I(n)$ equals the number of permutations of $[n]$ such that $I$ is the permutation's descent set. The polynomiality of $d_I(n)$ was first proven by MacMahon \cite{mcm}. %check it was macmahon

%"between descent sets and a class of skew young tableaux..." it's descent sets + n, rephrase%
Jiradilok and McConville \cite{jm19} show a bijection between descent sets as subsets of $[n-1]$  and ribbons (a class of Young diagrams) with exactly $n$ cells. As described in Section 2 below, this bijection equates the set of permutations of $[n]$ with a specific descent set to the set of skew tableaux of a specific skew shape. This allows Naruse's formula, discovered in \cite{naruse}, to be applied. Naruse's formula extends upon Frame, Robinson, and Thrall's Hook Length Formula \cite{frt} to determine the number of skew tableaux of skew shape $\lambda/\mu$. Applying this formula results in a simple expression for $d_I(n)$ as a summation of products of hook lengths across excited diagrams. 

%par 3
We write each descent polynomial through the Naruse-Newton basis (which has coefficients $C_i$) using the expression for $d_I(n)$ generated by Naruse's formula.  Jiradilok and McConville consider ratios between these \emph{Naruse-Newton coefficients} and prove that the minimum ratio of two coefficients $C_a$ and $C_b$ for $b>a$ is $\frac{a!}{b!}$ in Corollary 3.6 of \cite{jm19}. Results regarding these coefficients are used in bounding roots of descent polynomials, shown in \cite{jm19}'s proof of Conjecture 4.3 of \cite{dl}. 

%change last sentence of above paragraph to move [JM19]'s
%"used in the proof.... are manipulated" show that more clearly, kinda vague

We first extend upon Corollary 3.6 of \cite{jm19} in Proposition \ref{prop:strong-3.6}, and show it remains true for one additional index. We use this corollary to prove Theorem \ref{thm:num-min}, which finds all $a$ and $b$ for fixed $I$ such that $\frac{C_a}{C_b}$ is equal to the aforementioned minimum ratio. Following this result is Corollary \ref{cor:min-char}, which for nonnegative integers $b>a$, characterizes all descent sets in which the minimum ratio between $C_a$ and $C_b$ is attained.

%par 5
For nonnegative integers $b>a$, we examine $R_{a, b}$, defined as the set of all ratios $\frac{C_a}{C_b}$ over all descent sets.  In Theorem \ref{thm:no-max}, we show that this set is unbounded above, and provide a construction for which $\frac{C_a}{C_b}$ approaches infinity. This construction is used to prove that every point in $R_{a, b}$ is the limit of a convergent sequence of pairwise distinct points in $R_{a, b}$ in Corollary \ref{cor:closure}. Furthermore, for a nonnegative integer $a$ and fixed descent set with $s+1$ Naruse-Newton coefficients, we define $T_{a}$ as the set of all possible ratios $\frac{C_{s-a-1}}{C_{s-a}}$. We prove that the closure of $T_a$ is the set of nonnegative real numbers, shown in our Theorem \ref{thm:behind-closure}.

We also examine specific sequences of Naruse-Newton coefficients, namely those of two-element descent sets and those associated with ribbons of staircase shape. We find all unimodal and log-concave sequences of Naruse-Newton coefficients of two-element descent sets in Corollary \ref{cor:two-uni} and Corollary \ref{cor:two-log}, respectively. We then prove in Theorem \ref{thm:tripoly} that for all Naruse-Newton sequences associated with a ribbon of staircase shape and all nonnegative integers $i$, the ratio $\frac{C_{a-i}}{C_a}$ is a polynomial in $a$ of degree at most $i$. This lends to an explicit formula for all ratios of Naruse-Newton coefficients for ribbons of staircase shape. The proof of Theorem \ref{thm:tripoly} also incorporates a linear algebraic result as stated in Lemma \ref{lem:matrix}. In particular, we find a polynomial expression for the determinant of a specific matrix in two variables.

This paper is organized as follows. In Section 2, we introduce descent sets and Young diagrams, and detail the relation between the two. We also apply Naruse's formula to determine a hook length-based formula for $d_I(n)$, following the work of \cite{jm19}. Then, in Section 3, we define the Naruse-Newton basis and Naruse-Newton coefficients. In Section 4, we find all $a$ and $b$ for fixed $I$ such that $\frac{C_a}{C_b} = \frac{a!}{b!}$. We then characterize descent sets for fixed $a$ and $b$ in which $\frac{C_a}{C_b} = \frac{a!}{b!}$. In Section 5, we prove that $R_{a, b}$ lacks an upper bound, and examine the closures of $R_{a, b}$ and $T_a$. In Section 6, we examine properties of Naruse-Newton coefficients of doubletons, or descent sets with two elements. Finally, in Section 7, we examine properties of Naruse-Newton coefficients corresponding to ribbons of staircase shape.
\bigskip

\section{Preliminaries}

This section reviews descent polynomials, Young tableaux, Naruse's formula, the correspondence between descent sets and ribbons, and excitation factors. We develop our preliminaries in a similar fashion to \cite{jm19}.

Define a \emph{permutation} to be a bijection $\pi:[n] \rightarrow [n]$, where $[n] = \{1, 2, \ldots, n\}$. We use the one-line notation $\pi = \pi_1 \cdots \pi_n$ to mean $\pi(i) = \pi_i$. For each permutation $\pi$, define the \emph{descent set} of $\pi$ to be the set of $i$ in $[n-1]$ such that $\pi_i>\pi_{i+1}$.  Then, for each finite set $I$ of positive integers, define the \emph{descent polynomial} $d_I$ to be the unique univariate polynomial such that for all integers $n > \mathrm{max}(I \cup \{0\})$, the value $d_I(n)$ is equal to the number of permutations of $[n]$ with descent set $I$. MacMahon proved $d_I$ is a polynomial of degree $\mathrm{max}(I \cup \{0\})$ using the Principle of Inclusion-Exclusion \cite{mcm}.  

\begin{figure}
\centering
\begin{tikzpicture}
\begin{scope}
\draw (-3/4, 9/4) -- (-3/4, -3/4);
\draw (0, 9/4) -- (0, -3/4);
\draw (3/4, 9/4) -- (3/4, 0);
\draw (3/2, 9/4) -- (3/2, 3/4);
\draw (-3/4, 9/4) -- (3/2, 9/4);
\draw (-3/4, 3/2) -- (3/2, 3/2);
\draw (-3/4, 3/4) -- (3/2, 3/4);
\draw (-3/4, 0) -- (3/4, 0);
\draw (-3/4, -3/4) -- (0, -3/4);
\node at (-3/8, 15/8) {$c_{1, 1}$};
\node at (3/8, 15/8) {$c_{1, 2}$};
\node at (9/8, 15/8) {$c_{1, 3}$};
\node at (-3/8, 9/8) {$c_{2, 1}$};
\node at (3/8, 9/8) {$c_{2, 2}$};
\node at (9/8, 9/8) {$c_{2, 3}$};
\node at (-3/8, 3/8) {$c_{3, 1}$};
\node at (3/8, 3/8) {$c_{3, 2}$};
\node at (-3/8, -3/8) {$c_{4, 1}$};
\end{scope}

\begin{scope}[xshift=3.5cm]
\draw (-3/4, 9/4) -- (-3/4, -3/4);
\draw (0, 9/4) -- (0, -3/4);
\draw (3/4, 9/4) -- (3/4, 0);
\draw (3/2, 9/4) -- (3/2, 3/4);
\draw (-3/4, 9/4) -- (3/2, 9/4);
\draw (-3/4, 3/2) -- (3/2, 3/2);
\draw (-3/4, 3/4) -- (3/2, 3/4);
\draw (-3/4, 0) -- (3/4, 0);
\draw (-3/4, -3/4) -- (0, -3/4);
\node at (-3/8, 15/8) {\large 1};
\node at (3/8, 15/8) {\large 3};
\node at (9/8, 15/8) {\large 7};
\node at (-3/8, 9/8) {\large 2};
\node at (3/8, 9/8) {\large 5};
\node at (9/8, 9/8) {\large 8};
\node at (-3/8, 3/8) {\large 4};
\node at (3/8, 3/8) {\large 9};
\node at (-3/8, -3/8) {\large 6};
\end{scope}

\end{tikzpicture}
\caption{A straight Young diagram with labeled cells and a straight Young tableau, both of shape $(3, 3, 2, 1)$.}
\label{fig:youngdiagram}
\end{figure}

A \emph{partition} $\lambda$ is a weakly decreasing integer sequence $\lambda = (\lambda_1, \lambda_2, \ldots)$, with $\lambda_i \ge 0$ for all positive integers $i$, such that there exists an $N$ with $\lambda_N = 0$. For each partition $\lambda$, the \emph{straight Young diagram} $\mathbb{D}(\lambda)$ of shape $\lambda$ is defined as the left-justified collection of cells containing $\lambda_i$ cells in its $i$th row, for positive integer $i$. We use the English notation, where the row corresponding to $\lambda_1$ is at the top of the diagram. The \emph{conjugate partition} $\lambda'$ of the partition $\lambda$ is defined such that there are $\lambda'_i$ cells in the $i$th column of $\mathbb{D}(\lambda)$. Let $|\lambda|$ denote the number of cells in $\mathbb{D}(\lambda)$. Denote $c_{i, j}$ as the cell in the $i$th row and $j$th column. In Figure \ref{fig:youngdiagram}, we provide a straight Young diagram with $c_{i, j}$s labeled and a straight Young tableau, both of shape $(3, 3, 2, 1)$.

For a partition $\lambda$ with $n = |\lambda|$, define a \emph{straight Young tableau} of shape $\lambda$ to be a bijective filling of the cells of the Young diagram $\mathbb{D}(\lambda)$ with $[n]$, such that the value of each cell is less than those of all cells below it and to its right. 

%par 5
A formula for $f^\lambda$, defined as the number of Young tableaux of shape $\lambda$, was discovered by Frame, Robinson, and Thrall in 1954 \cite{frt}. The elegant formula is based upon the product of hook lengths of cells, where the \emph{hook length} of a cell $c$, denoted as $h_\lambda(c)$, is equal to the number of cells weakly below or weakly to the right of $c$ (including $c$ itself). 

\begin{theorem}[Hook Length Formula \cite{frt}] \label{thm:hlf} 
For a partition $\lambda$ with $n = |\lambda|$,
$$f^\lambda = n!\displaystyle\prod_{c \in \mathbb{D}(\lambda)} \frac{1}{h_\lambda(c)}.$$
\end{theorem}

Naruse's Formula generalizes the Hook Length Formula, considering skew Young tableaux rather than straight Young tableaux. However, before we introduce Naruse's formula, we must introduce a few structures that are used throughout this paper.

For two partitions $\lambda$ and $\mu$ such that $\lambda_i\ge \mu_i$ for all positive integers $i$, define the \emph{skew Young diagram} $\mathbb{D}(\lambda/\mu)$ of shape $\lambda/\mu$ to be the collection of cells in $\mathbb{D}(\lambda)$ that are not in $\mathbb{D}(\mu)$.

For a skew Young diagram $\mathbb{D}(\lambda/\mu)$ with $n = |\lambda|-|\mu|$, define a \emph{skew Young tableau} to be a bijective filling of the cells of $\mathbb{D}(\lambda/\mu)$ with $[n]$, such that the value of each cell is less than those of all cells below it and to its right. This is akin to a straight Young tableau, but involving a skew Young diagram rather than a straight Young diagram. Examples of skew Young tableaux of shape $(3, 3, 2, 1)/(2, 1)$ are shown in Figure \ref{fig:ribbon1}.

Naruse discovered a formula for the number of skew Young tableaux of shape $\lambda/\mu$ similar to the Hook Length Formula \cite{naruse}, which involves structures known as excited diagrams. For a skew Young diagram of shape $\lambda/\mu$, an \emph{excited diagram} is a subset of cells $D \subseteq \mathbb{D}(\lambda)$ characterized by repeated applications of a particular operation, described in the following sentence, on the cells in $\mathbb{D(\mu)}$. This particular operation can be applied to an excited diagram $D$ by replacing a cell $c_{i, j} \in D$ with cell $c_{i+1, j+1}$ if and only if $c_{i+1, j+1} \in \mathbb{D(\lambda)}$ and $\{c_{i, j+1}, c_{i+1, j}, c_{i+1, j+1}\} \cap D = \varnothing$. For example, there are five excited diagrams of the skew Young diagram of shape $(3, 3, 2, 1)/(2, 1)$, as shown in Figure \ref{fig:ribbon2}.

\begin{figure}
\centering
\begin{tikzpicture}
\begin{scope}
\draw[draw=black, fill=gray] (-1/2, 3/2) rectangle (0, 1);
\draw[draw=black, fill=gray] (1/2, 3/2) rectangle (0, 1);
\draw (1, 3/2) rectangle (1/2, 1);
\node at (3/4, 5/4) {4};
\draw[draw=black, fill=gray] (-1/2, 1) rectangle (0, 1/2);
\draw (0, 1) rectangle (1/2, 1/2);
\draw (1, 1) rectangle (1/2, 1/2);
\node at (3/4, 3/4) {6};
\node at (1/4, 3/4) {1};
\draw (-1/2, 1/2) rectangle (0, 0);
\draw (1/2, 1/2) rectangle (0, 0);
\node at (1/4, 1/4) {3};
\node at (-1/4, 1/4) {2};
\draw (-1/2, 0) rectangle (0, -1/2);
\node at (-1/4, -1/4) {5};
\end{scope}

\begin{scope}[xshift=2.5cm]
\draw[draw=black, fill=gray] (-1/2, 3/2) rectangle (0, 1);
\draw[draw=black, fill=gray] (1/2, 3/2) rectangle (0, 1);
\draw (1, 3/2) rectangle (1/2, 1);
\node at (3/4, 5/4) {3};
\draw[draw=black, fill=gray] (-1/2, 1) rectangle (0, 1/2);
\draw (0, 1) rectangle (1/2, 1/2);
\draw (1, 1) rectangle (1/2, 1/2);
\node at (3/4, 3/4) {5};
\node at (1/4, 3/4) {2};
\draw (-1/2, 1/2) rectangle (0, 0);
\draw (1/2, 1/2) rectangle (0, 0);
\node at (1/4, 1/4) {6};
\node at (-1/4, 1/4) {1};
\draw (-1/2, 0) rectangle (0, -1/2);
\node at (-1/4, -1/4) {4};
\end{scope}

\begin{scope}[xshift=5cm]
\draw[draw=black, fill=gray] (-1/2, 3/2) rectangle (0, 1);
\draw[draw=black, fill=gray] (1/2, 3/2) rectangle (0, 1);
\draw (1, 3/2) rectangle (1/2, 1);
\node at (3/4, 5/4) {5};
\draw[draw=black, fill=gray] (-1/2, 1) rectangle (0, 1/2);
\draw (0, 1) rectangle (1/2, 1/2);
\draw (1, 1) rectangle (1/2, 1/2);
\node at (3/4, 3/4) {6};
\node at (1/4, 3/4) {3};
\draw (-1/2, 1/2) rectangle (0, 0);
\draw (1/2, 1/2) rectangle (0, 0);
\node at (1/4, 1/4) {4};
\node at (-1/4, 1/4) {1};
\draw (-1/2, 0) rectangle (0, -1/2);
\node at (-1/4, -1/4) {2};
\end{scope}

\end{tikzpicture}
\caption{$3$ of the $61$ skew tableaux of shape $(3, 3, 2, 1)/(2, 1)$.}
\label{fig:ribbon1}
\end{figure}

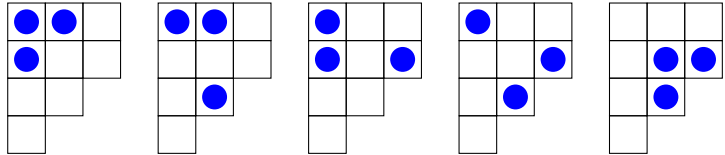
\begin{figure}
\centering
\begin{tikzpicture}
\begin{scope}
\draw (-1/2, 3/2) rectangle (0, 1);
\draw (1/2, 3/2) rectangle (0, 1);
\draw (1, 3/2) rectangle (1/2, 1);
\draw (-1/2, 1) rectangle (0, 1/2);
\draw (0, 1) rectangle (1/2, 1/2);
\draw (1, 1) rectangle (1/2, 1/2);
\draw (-1/2, 1/2) rectangle (0, 0);
\draw (1/2, 1/2) rectangle (0, 0);
\draw (-1/2, 0) rectangle (0, -1/2);
\node[shape=circle, fill=blue] at (-1/4, 5/4){};
\node[shape=circle, fill=blue] at (1/4, 5/4){};
\node[shape=circle, fill=blue] at (-1/4, 3/4){};
\end{scope}

\begin{scope}[xshift=2cm]
\draw (-1/2, 3/2) rectangle (0, 1);
\draw (1/2, 3/2) rectangle (0, 1);
\draw (1, 3/2) rectangle (1/2, 1);
\draw (-1/2, 1) rectangle (0, 1/2);
\draw (0, 1) rectangle (1/2, 1/2);
\draw (1, 1) rectangle (1/2, 1/2);
\draw (-1/2, 1/2) rectangle (0, 0);
\draw (1/2, 1/2) rectangle (0, 0);
\draw (-1/2, 0) rectangle (0, -1/2);
\node[shape=circle, fill=blue] at (-1/4, 5/4){};
\node[shape=circle, fill=blue] at (1/4, 5/4){};
\node[shape=circle, fill=blue] at (1/4, 1/4){};
\end{scope}

\begin{scope}[xshift=4cm]
\draw (-1/2, 3/2) rectangle (0, 1);
\draw (1/2, 3/2) rectangle (0, 1);
\draw (1, 3/2) rectangle (1/2, 1);
\draw (-1/2, 1) rectangle (0, 1/2);
\draw (0, 1) rectangle (1/2, 1/2);
\draw (1, 1) rectangle (1/2, 1/2);
\draw (-1/2, 1/2) rectangle (0, 0);
\draw (1/2, 1/2) rectangle (0, 0);
\draw (-1/2, 0) rectangle (0, -1/2);
\node[shape=circle, fill=blue] at (-1/4, 5/4){};
\node[shape=circle, fill=blue] at (3/4, 3/4){};
\node[shape=circle, fill=blue] at (-1/4, 3/4){};
\end{scope}

\begin{scope}[xshift=6cm]
\draw (-1/2, 3/2) rectangle (0, 1);
\draw (1/2, 3/2) rectangle (0, 1);
\draw (1, 3/2) rectangle (1/2, 1);
\draw (-1/2, 1) rectangle (0, 1/2);
\draw (0, 1) rectangle (1/2, 1/2);
\draw (1, 1) rectangle (1/2, 1/2);
\draw (-1/2, 1/2) rectangle (0, 0);
\draw (1/2, 1/2) rectangle (0, 0);
\draw (-1/2, 0) rectangle (0, -1/2);
\node[shape=circle, fill=blue] at (-1/4, 5/4){};
\node[shape=circle, fill=blue] at (3/4, 3/4){};
\node[shape=circle, fill=blue] at (1/4, 1/4){};
\end{scope}

\begin{scope}[xshift=8cm]
\draw (-1/2, 3/2) rectangle (0, 1);
\draw (1/2, 3/2) rectangle (0, 1);
\draw (1, 3/2) rectangle (1/2, 1);
\draw (-1/2, 1) rectangle (0, 1/2);
\draw (0, 1) rectangle (1/2, 1/2);
\draw (1, 1) rectangle (1/2, 1/2);
\draw (-1/2, 1/2) rectangle (0, 0);
\draw (1/2, 1/2) rectangle (0, 0);
\draw (-1/2, 0) rectangle (0, -1/2);
\node[shape=circle, fill=blue] at (1/4, 3/4){};
\node[shape=circle, fill=blue] at (3/4, 3/4){};
\node[shape=circle, fill=blue] at (1/4, 1/4){};
\end{scope}

\end{tikzpicture}
\caption{The $5$ excited diagrams of the skew Young diagram of shape $(3, 3, 2, 1)/(2, 1)$.}
\label{fig:ribbon2}
\end{figure}

Denote the set of excited diagrams for a skew Young diagram of shape $\lambda/\mu$ as $\mathbb{E}(\lambda/\mu)$. Naruse's formula for the number of skew Young tableaux of shape $\lambda/\mu$, denoted by $f^{\lambda/\mu}$, is as follows.

\begin{theorem}[Naruse's formula \cite{naruse}] \label{thm:naruse}
For a skew diagram of shape $\lambda/\mu$ with $n = |\lambda|-|\mu|$,
$$f^{\lambda/\mu} = \frac{n!}{\displaystyle\prod_{c \in \mathbb{D}(\lambda)} h_\lambda(c)}\displaystyle\prod_{D \in \mathbb{E}(\lambda/\mu)}\displaystyle\sum_{d \in D} h_\lambda(d).$$
\end{theorem}

For the aforementioned skew Young diagram of shape $(3, 3, 2, 1)/(2, 1)$, there are $$\frac{6!}{6\cdot 5 \cdot 4 \cdot 3^2 \cdot 2 \cdot 1^3} \left(6 \cdot 5 \cdot 4 + 6 \cdot 4 \cdot 1 + 6 \cdot 5 \cdot 1 + 6 \cdot 1^2 + 3 \cdot 1^2\right) = 61$$
distinct tableaux, shown through an application of Naruse's formula.

For a partition $\lambda$, the \emph{ribbon} $\mathbb{D}_{\operatorname{rib}}(\lambda) \subseteq \mathbb{D}(\lambda)$ is the skew Young diagram such that a cell $c \in \mathbb{D}(\lambda)$ is contained in $\mathbb{D}_{\operatorname{rib}}(\lambda)$ if and only if $|\{c_{i, j+1}, c_{i+1, j}, c_{i+1, j+1}\} \cap \mathbb{D}(\lambda)| < 3$. For example, as shown in Figure \ref{fig:ribbon1}, the skew Young diagram of shape $(3, 3, 2, 1)/(2, 1)$ is the ribbon $\mathbb{D}_{\operatorname{rib}}((3, 3, 2, 1))$. Note that $\mathbb{D}_{\operatorname{rib}}$ is a one-to-one correspondence between partitions and ribbons.

Define a \emph{tableau of ribbon shape} to be a skew Young tableau whose shape is a ribbon. For a tableau of ribbon shape with shape $\lambda/\mu$ and size $n = \lambda-\mu$, the filling of the entries determines a permutation $\pi = \pi_1 \cdots \pi_n$, where the entry of the lower left corner cell is $\pi_1$, the entry of the cell adjacent to the lower left corner cell is $\pi_2$, and successive entries move up and to the right until the upper right corner cell, whose entry is $\pi_n$. Observe that the descent set of $\pi$ is equal to the set of positions $i$ such that the cell containing $\pi_i$ is directly below the cell containing $\pi_{i+1}$. This descent set does not depend on the permutation, it only depends on the shape of the ribbon. In other words, each ribbon of size $n$ must have a unique corresponding descent set. Furthermore, for fixed $n$ and descent set $I \subseteq [n-1]$, there exists a unique ribbon of size $n$ corresponding to $I$, as the presence (or absence) of an element in $I$ fixes the location of the cell in the corresponding position of the ribbon. This implies a bijection between the ribbons of size $n$ and the subsets of $[n-1]$. For example, under this bijection, the ribbon shown in Figure \ref{fig:ribbon1} would correspond to the descent set $\{1, 3, 5\} \subseteq [5]$.

Define $\lambda^I$ to be the unique partition such that the ribbon $\mathbb{D}_{\operatorname{rib}}(\lambda^I)$ corresponds to descent set $I$ in accordance to the bijection above, and $(\lambda^I)_1 = (\lambda^I)_2$. Partition $\mu^I$ is defined such that $\mathbb{D}_{\operatorname{rib}}(\lambda^I) = \mathbb{D}(\lambda^I/\mu^I)$.

Let $m \ge 2$ be a positive integer, and take a set $I$ of positive integers such that $\mathrm{max}(I)=m-1$.  For partition $\lambda$ and positive integer $t$, define the partition $\lambda^t = (\lambda_1+t, \lambda_2, \lambda_3, \ldots)$. Take $n=m+t$ for variable nonnegative integer $t$, then the value $d_I(n)$ is the number of tableaux of ribbon shape with shape $(\lambda^{I})^t/\mu^I$. We use this property, along with Naruse's formula, to discover a nice formula for $d_I$.

Note that for all sets of positive integers $I$ and nonnegative integers $t$, the ribbons of shape $\lambda^I/\mu^I$ and $(\lambda^I)^{t}/\mu^I$ have the same set of excited diagrams. Furthermore, the only cells whose hook lengths vary with $n$ must be in the first row, and the set of cells in the first row in an excited diagram must be in the form $\{c_{1, 1}, c_{1, 2}, \ldots, c_{1, k}\}$. For a subset of cells $D \subseteq \mathbb{D}((\lambda^I)^{t}/\mu^I)$, let $\overline{D}$ be the subset created by removing all cells in $D$ of the form $c_{1, i}$. Define the \emph{hook product} of a diagram as the product of hook lengths of all of its cells. Then, let $E(t)$ denote the \emph{excitation factor} of $(\lambda^I)^t/\mu^I$, or the sum of the hook products of each excited diagram. Additionally, fix $\alpha_i = h_{\lambda^I}(c_{1, i})-1$. We now compute $E(t)$, which is necessary to apply Naruse's formula to $\mathbb{D}((\lambda^I)^t/\mu^I)$. We achieve this by taking the summation of hook products over all excited diagrams $D \subseteq \mathbb{E}((\lambda^I)^t/\mu^I)$. In particular, we split the hook product of each excited diagram into two sub-products: the first concerning cells of the form $c_{1, i}$, the second concerning all other cells. Thus, we have
\begin{equation} \label{excite}  E(t) = \displaystyle\sum_{D \in \mathbb{E}((\lambda^I)^t/\mu^I)} \left(\displaystyle\prod_{c \in \overline{D}} h_{\lambda^I}(c)\right)\left(\displaystyle\prod_{c_{1, i} \in D}(t+\alpha_i)\right). \end{equation}

Denote $n = m+t$ for variable $t$. By applying Naruse's formula to $\mathbb{D}((\lambda^I)^t/\mu^I)$, we now compute $f^{(\lambda^I)^t/\mu^I}$, which is equal to $d_I(n)$. Thus,
\begin{align} 
d_I(n) &= d_I(m+t)\nonumber \\
\label{des-poly}
 &= (m+t)!\left(\displaystyle\prod_{c \in \overline{\mathbb{D}(\lambda^I/\mu^I)}}\frac{1}{h_\lambda(c)}\right)\left(\displaystyle\prod_{i=1}^{\lambda_1} \frac{1}{t+\alpha_i}\right)\frac{1}{(t-1)!}E(t).
\end{align}

The terms in the expression of $d_I(n)$ other than $E(t)$ represent the quotient of $n!$ by the hook product of $\mathbb{D}(\lambda^I)$. We examine $E(t)$ for its combinatorial significance, as the other terms in Equation \eqref{des-poly} are already understood as a product of monomials. The next section simplifies $E(t)$ into linear combinations of terms of the Naruse-Newton basis.
\bigskip

\section{Naruse-Newton Coefficients}

The Naruse-Newton basis was introduced in \cite{jm19}, and our paper focuses on the related Naruse-Newton coefficients. This section serves as an introduction to the basis and its coefficients.

Consider a descent set $I$ and its corresponding class of ribbons $\mathbb{D}_{\operatorname{rib}}((\lambda^{I})^t)$ for a nonnegative integer $t$ (refer to Section 2 for a detailed explanation of the correspondence). Denote $s = (\mu^I)_1$. Recall Equation \eqref{excite} as the formula for the excitation factor.  Each $c_{1, k}$ has a hook length of $(t+\alpha_k+1)$, which is linear in $t$. Meanwhile, each $c_{i, j}$ with $i>1$ is not dependent upon $t$. The initial excited diagram $\mathbb{D}(\mu^I)$ contains $s$ cells of the form $c_{1, k}$. Thus, the $t$-degree of the term in $E(t)$ corresponding to $\mathbb{D}(\mu^I)$, which equals $$\left(\displaystyle\prod_{c \in \overline{\mathbb{D}(\mu^I)}} h_{\lambda^I}(c)\right)\left(\displaystyle\prod_{c_{1, i} \in \mathbb{D}(\mu^I)}(t+\alpha_i)\right),$$ is equal to $s$. Since there do not exist any excited diagrams with more than $s$ cells of the form $c_{1, k}$, the degree of polynomial $E(t)$ must be equal to $s$.

For variable $x$, a \emph{Newton basis} is a sequence of polynomials $$\left(a_k\displaystyle\prod_{i=1}^k(x+b_i)\right)_{k=0}^m$$ for nonnegative integer $m$ and complex sequences $(a_k)_{k=0}^m$ and $(b_k)_{k=0}^m$, such that each $a_k$ is nonzero. Note that a Newton basis is a $\mathbb{C}$-linear basis for the $(m+1)$-dimensional vector space $\mathbb{C}[x]^{\le m}$ of polynomials of degree at most $m$. Recall from Equation \eqref{excite} that each summand of the polynomial $E(t)$ can be written as $p \left(\displaystyle\prod_{i=1}^j (t+\alpha_i)\right)$ for nonnegative integers $p$ and $j$.

In Equation \eqref{excite}, the excitation factor $E(t)$ is expressed through the Newton basis consisting of polynomials $1, t+\alpha_1, (t+\alpha_1)(t+\alpha_2), \ldots, (t+\alpha_1)\cdots(t+\alpha_s)$, defined as the \emph{Naruse-Newton basis}. Nonnegative integers $C_0, C_1, \ldots, C_s$, which we call the \emph{Naruse-Newton coefficients} of the descent set $I$, are defined such that $$E(t) = C_0 (t+\alpha_1)\cdots(t+\alpha_s) + \cdots + C_{s-1}(t+\alpha_1)+C_s.$$

We also denote $C_j$ of descent set $I$ as $C_j(I)$. This formulation of the excitation factor serves a distinct advantage. For nonnegative integer $i \le s$, we have that $C_{s-i}(t+\alpha_1)\cdots(t+\alpha_i)$ is the sum of the hook products over all excited diagrams of $\mathbb{D}((\lambda^I)^t/\mu^I)$ with $i$ cells in the first row of $\lambda$.

For an example of the sequence of $C_i$, take the ribbon of shape $\mathbb{D}_{\operatorname{rib}}(\lambda)$ where $\lambda = (3, 3, 2, 1)$. Its excited diagrams are drawn in Figure \ref{fig:ribbon2}. Then, we know $C_0$ equals the sum of hook products of cells outside the first row across the set of excited diagrams containing $c_{1, 1}$ and $c_{1, 2}$. These are the leftmost two diagrams. Hence, we compute $C_0 = 5+1 = 6$. Similarly, the third and fourth diagrams with one cell in the first row are counted in the computation of $C_1$, and the fifth diagram with no cells in the first row is counted in the computation of $C_2$. Finally, we compute $C_1 = 5 \cdot 1 + 1 \cdot 1 = 6$, and $C_2 = 1 \cdot 3 \cdot 1 = 3$. Additional examples of Naruse-Newton coefficients can be found in Appendix A.

As shown in the example, the Naruse-Newton coefficients measure the ``relative contribution" to the excitation factor of each class of excited diagrams. Jiradilok and McConville use bounds between coefficients in \cite{jm19} to prove a conjecture relating to descent polynomials, and a key goal of this paper is to expand upon those bounds.

\bigskip
\section{Characterization of Minimum Ratios}

For a descent set $I$ and integers $a$ and $b$ such that $0 \le a<b \le s$, define the ratio $C_{a, b}(I) = \frac{C_a(I)}{C_b(I)}$, where $C_a$ and $C_b$ are the $a$th and $b$th Naruse-Newton coefficients corresponding to $I$.

In this section, we recall the minimum bound to the ratio $C_{a, b}(I)$ proven in \cite{jm19}, shown in  Proposition \ref{prop:jm19-2.4}. We then find all equality cases to this bound. To do so, we first monotonically change ratios by manipulating the ribbon $\mathbb{D}_{\operatorname{rib}}(\lambda^I)$, shown in Proposition \ref{prop:add-row}. A monotonic decrease of ratios in an infinite sequence of descent sets implies that $C_{a, b}$ is not minimized for any of these sets, allowing us to find the minimum bound. Thus, as shown in Theorems \ref{thm:num-min} and \ref{cor:min-char}, we can find all $a$ and $b$ for fixed $I$ such that $C_{a, b}(I)$ is minimized, and also characterize all sets for fixed $a$ and $b$ for which $C_{a, b}$ is minimized.

\begin{proposition}[{\cite[Proposition~2.4]{jm19}}] \label{prop:jm19-2.4}
Let $(C_0, C_1, \ldots, C_s)$ be the Naruse-Newton coefficients corresponding to a fixed non-empty descent set $I$. Then, $$\frac{C_0}{0!} \ge \frac{C_1}{1!} \ge \cdots \ge \frac{C_s}{s!}.$$
\end{proposition}

This proposition is proven in \cite{jm19} using the Slice and Push Inequality. Thus, for any $0 \le a < b \le s$, we have $C_{a, b} \ge \frac{a!}{b!}$. The equality case is particularly insightful, as it provides a sharp lower bound to this ratio. We now characterize descent sets which achieve this equality case.

\begin{definition} \label{def:deep}
Let $I$ be a non-empty set of positive integers, and let $s = (\lambda^I)_1-1$. We call $I$ a \emph{deep descent set} if the number of distinct $i$ such that $1 \le i \le s+1$ and $\lambda'_i=2$ is less than $s$. 

On the other hand, if $I$ is not a deep descent set, then it is a \emph{shallow descent set}.
\end{definition}

In other words, a non-empty set $I$ of positive integers is a deep descent set if and only if $\mathbb{D}_{\operatorname{rib}}(\lambda^I)$ has at least two columns and $(\lambda^I)'_2>2$. 

For a descent set $I$, let $w(I)$ be the nonnegative integer equal to the number of distinct $1 \le i \le s+1$ such that $\lambda'_i = 2$. For a ribbon $\mathbb{D}_{\operatorname{rib}}(\lambda^I)$, define $\mathbb{D}_{\operatorname{rib}}(\lambda^I[i])$ to be the ribbon representing the partition $((\lambda^I)_{i+1}, (\lambda^I)_{i+2}, (\lambda^I)_{i+3}, \ldots)$, in other words, the ribbon formed by removing the topmost $i$ rows from $\mathbb{D}_{\operatorname{rib}}(\lambda^I)$.

We first prove the following case obtains the minimum ratio. Note that the next proposition is virtually identical to, except one index stronger than Corollary 3.6 of \cite{jm19}.

\begin{proposition} \label{prop:strong-3.6}
Let $I$ be a non-empty set of positive integers, and define $\lambda=\lambda^I$ where $\lambda^I$ is defined in Section 2.  Let $s = \lambda_1-1$, and let $(C_0, C_1, \ldots, C_s)$ be the Naruse-Newton coefficients of $I$. If for some positive integer $i \le s$ it holds true that $\lambda'_{i+1} = \lambda'_{s+1} = 2$, then  $$\frac{C_0}{0!} = \frac{C_1}{1!} = \cdots = \frac{C_{s-i+1}}{(s-i+1)!}.$$
\end{proposition}

\begin{proof}
It is sufficient to prove that $C_0 \cdot (s-i+1)! = C_{s-i+1}$ because of Proposition \ref{prop:jm19-2.4}. Recall that $C_0$ is equal to the excitation factor of ${\mathbb{D}}_{\operatorname{rib}}(\lambda^I[1])$. In addition, we have that $C_{s-i+1}$ is the product of $\displaystyle\prod_{k=i+1}^{s+1} h_{\lambda^I}\left(c_{2, k}\right) = (s-i+1)!$ with the excitation factor of $\mathbb{D}_{\operatorname{rib}}(\lambda^I[1])$. This is as desired, and the proof is complete.
\end{proof}

It is clear from Proposition \ref{prop:strong-3.6} that all shallow descent sets satisfy $\frac{C_0}{0!} = \frac{C_1}{1!} = \cdots = \frac{C_s}{s!}$. As shallow descent sets obtain the minimum value for all $C_{a, b}$ with $a<b$, we now only need to focus on deep descent sets (rather than both deep and shallow ones). To do this, we begin by focusing on monotonic changes of the ratio $C_{a, b}$, in particular, changes to the value of $C_{a, b}$ when $\mathbb{D}_{\operatorname{rib}}(\lambda^I)$ is manipulated.

For a non-empty descent set $I$, define $I- = I-\max(I)$. We define sequences of nonnegative integers $\{D_i\}$ and $\{E_i^j\}$ corresponding to set $I$, with $D_i$ and $E_i^j$ also expressed as $D_i(I)$ and $E_i^j(I)$. The sequence $\{D_i\}$ is defined for integer $i$ with $s \ge i \ge w$ such that $D_i(I) = C_{i-w(I)}(I-)$. For example, in Figure \ref{fig:ribbon3}, the ribbon $\mathbb{D}_{\operatorname{rib}} (\lambda^I)$ is outlined and the ribbon $\mathbb{D}_{\operatorname{rib}} (\lambda^{I-})$ is bolded for $I = \{2, 4, 6\}$. In this case, we have that $D_i(I) = C_{i-2}(I-)$. The sequence $\{E_i^j\}$ is defined for integers $i$ and $j$ with $\lambda'_1 \ge i \ge 0$ and $s \ge j \ge 0$, such that $E_i^j$ represents the excitation factor of the ribbon formed by removing the rightmost $j$ columns from $\mathbb{D}_{\operatorname{rib}}(\lambda^I[i])$. In other words, $E_i^j$ is the excitation factor of the ribbon formed by removing the rightmost $j$ columns and topmost $i$ rows from $\mathbb{D}_{\operatorname{rib}}(\lambda^I)$.

\tikzset{
    ultra ultra thick/.style={line width=3pt}
}

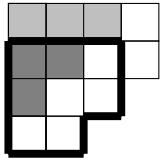
\begin{figure} 
\centering
\begin{tikzpicture}
\draw[draw=black, fill=light-gray] (-1/2, 3/2) rectangle (0, 1);
\draw[draw=black, fill=light-gray] (1/2, 3/2) rectangle (0, 1);
\draw[draw=black, fill=light-gray] (-1, 3/2) rectangle (-1/2, 1);
\draw[draw=black, fill=gray] (-1, 2/2) rectangle (-1/2, 1/2);
\draw[draw=black, fill=gray] (-1, 1/2) rectangle (-1/2, 0);
\draw (1, 3/2) rectangle (1/2, 1);
\draw[draw=black, fill=gray] (-1/2, 1) rectangle (0, 1/2);
\draw (0, 1) rectangle (1/2, 1/2);
\draw (1, 1) rectangle (1/2, 1/2);
\draw (-1/2, 1/2) rectangle (0, 0);
\draw (1/2, 1/2) rectangle (0, 0);
\draw (-1/2, 0) rectangle (0, -1/2);
\draw (-1, 0) rectangle (-1/2, -1/2);
\draw (-1, 3/2) rectangle (-1/2, -1/2);
\draw (-1, 0) rectangle (-1/2, -1/2);
\draw[ultra ultra thick] (-1, 1) rectangle (1/2, 1);
\draw[ultra ultra thick] (1/2, 1) rectangle (1/2, 0);
\draw[ultra ultra thick] (0, 0) rectangle (1/2, 0);
\draw[ultra ultra thick] (0, -1/2) rectangle (0, 0);
\draw[ultra ultra thick] (0, -1/2) rectangle (-1, -1/2);
\draw[ultra ultra thick] (-1, -1/2) rectangle (-1, 1);
\end{tikzpicture}
\caption{$\mathbb{D}_{\operatorname{rib}} (\lambda^I)$ outlined and $\mathbb{D}_{\operatorname{rib}} (\lambda^{I-})$ bolded for $I = \{2, 4, 6\}$.}
\label{fig:ribbon3}
\end{figure}

\begin{proposition} \label{prop:add-row}
Let $I$ be a non-empty deep descent set of positive integers, and define $\lambda = \lambda^I$ where $\lambda^I$ is defined in Section 2. Let $s = \lambda_1-1$ and $(C_0, C_1, \ldots, C_s)$ be the Naruse-Newton coefficients of $I$, and let $w(I)$ be defined as above. Let positive integers $a$ and $b$ be defined such that $s \ge b > w(I)$ and $b > a \ge 0$, and let $J = \{i+1 \mid i \in I\}$. Then, it holds that $C_{a, b}(J)<C_{a, b}(I)$.
\end{proposition}

\begin{proof}
It is sufficient to prove that $C_{i, i+1}(I)>C_{i, i+1}(J)$ for all $s > i \ge w(I)$, as $C_{a, b}(I) = \displaystyle\prod_{i=a}^{b-1}C_{i, i+1}(I)$.

Our proof is structured as follows: We examine the ratio $\frac{C_k}{C_{k+1}}$ for a fixed nonnegative integer $k$ with $s>k \ge w(I)$.
We begin by expressing $C_{k}$ and $C_{k+1}$ in terms of $D_i$ and $E_i^j$. We can then examine changes between terms in the ratios $C_{k, k+1}(I)$ and $C_{k, k+1}(J)$, concluding via induction that $C_{k, k+1}(I)>C_{k, k+1}(J)$.

For $i \in [\lambda_1]$, define $H_i = h_{\lambda^I}(c_{2, i})$ and for $i \in [\lambda'_{s-k+1}]$, define $H'_i = h_{\lambda^I}(c_{i, s-k+1})$.  These are defined with respect to descent set $I$, and are also expressed as $H_i(I)$ and $H'_i(I)$.

Let $v = \lambda'_{s+1-k}-2$ and let $G$ denote the hook product of all cells of $\mathbb{D}_{\operatorname{rib}}(\lambda^I)$ weakly below and to the right of $C_{2, s-k+2}$. Recall the definitions for sequences $\{D_i\}$ and $\{E_i^j\}$. We begin by finding $C_k$ by examining the set of excited diagrams of $\mathbb{D}_{\operatorname{rib}}(\lambda^I)$ which contain $s-k$ cells in their first row. We can partition this set into two subsets: Excited diagrams which contain $c_{2, s-k}$, and excited diagrams which do not. Thus, 
$$C_k = G \cdot E^{k+2}_1 \cdot \displaystyle\prod_{i=2}^{v+1} H'_{i+1} + D_k \cdot \frac{\displaystyle\prod_{i=1}^{s+1}H_i}{h_{s-k+1}}.$$
Furthermore, we can express $C_{k+1}$ in terms of $E^{k+2}_1$ by separating cells into those to the left of the $(s-k)$th row, those in the $(s-k)$th row, and those to the right of the $(s-k)$th row. We now have $$C_{k+1} = G \cdot E^{k+2}_1 \cdot  \displaystyle\prod_{i=1}^{v+1}H'_{i+1}.$$ 
We can then expand $E^{k+2}_1$ into a sequence involving some $D_i$ by expressing it in a Naruse-Newton basis, similar to the process in Section 3. Thus, $$E^{k+2}_1 = \frac{\displaystyle\prod_{i=s-k+2}^{s+1}H_i}{G}\left(D_{k+1}\left(\displaystyle\prod_{i=1}^{s-k-1}H_i\right) + D_{k+2}\left(\displaystyle\prod_{i=1}^{s-k-2}H_i \right) + \cdots +D_s \right).$$

Through these previous expressions, we can now take the ratio
\begin{align*}
\frac{C_k}{C_{k+1}} &= \frac{G\cdot E^{k+2}_1 \cdot \displaystyle\prod_{i=2}^{v+1} H'_{i+1} + D_k \cdot \frac{\displaystyle\prod_{i=1}^{s+1}H_i}{h_{s-k+1}}}{G\cdot E^{k+2}_1 \cdot \displaystyle\prod_{i=1}^{v+1} H'_{i+1}}
 \\ &= \frac{1}{H'_{2}} + \frac{D_k \cdot \displaystyle\prod_{i=1}^{s-k}H_i}{\left(\displaystyle\prod_{i=1}^{v+1}H'_{i+1}\right)\displaystyle\left(\sum_{i=k+1}^s \left(D_i \cdot \displaystyle\prod_{j=1}^{s-i}H_j \right)\right)}.
\end{align*}

We proceed by induction. For some integer $h$, assume $C_{i, i+1}(J) < C_{i, i+1}(I)$ for all deep descent sets with $(\lambda^I)_1 \le h$. We will prove $C_{i, i+1}(J) < C_{i, i+1}(I)$ for all deep descent sets with $(\lambda^I)_1= h+1$. The base case occurs when $h=3$, which we will show the following process continues to hold true. Note that since $\mathbb{D}_{\operatorname{rib}}(\lambda^I[1])$ is a shallow descent set, $\frac{D_i(J)}{D_k(J)}=\frac{D_i(I)}{D_k(I)}$.

Take a descent set $I$ such that  $\mathbb{D}_{\operatorname{rib}}(\lambda^I)$ with height $h+1$. To form this ribbon from a ribbon of height $h$, we take the ribbon consisting of the cells weakly below $\lambda_2$ and inclusively between $\lambda'_1$ and $\lambda'_{s+1-w}$. Then, add a row to the top of this ribbon, and add a $w(I) \times 2$ rectangle of cells to its upper right corner.

To maintain uniformity with $\mathbb{D}_{\operatorname{rib}}(\lambda^I)$, we re-number the columns of $\mathbb{D}_{\operatorname{rib}}(\lambda^J)$ from $0$ to $s+1$. Thus, for $i \in [\lambda_1]$, we have that $H_i(I) = H_i(J)$, and for $i \in [\lambda'_{s-k+1}]$, we have that $H'_i(I) = H'_i(J)$. Furthermore, the value  $v=\lambda'_{s+1-k}-2$ remains the same for both ribbons. We now compare specific terms to show $C_{i, i+1}(J) < C_{i, i+1}(I)$. First, for all $i>k$, our assumption implies that $$\frac{D_i(J)}{D_k(J)}>\frac{D_i(I)}{D_k(I)}.$$

Since for all $j$ satisfying $\lambda_1>j \ge 1$, it holds that $H_j(I) = H_j(J)$, we can then sum the previous expression across all $i$ for $s\ge i\ge k+1$ to generate the inequality $$\frac{\displaystyle\sum_{i=k+1}^s\left( D_i(J)\displaystyle\prod_{j=1}^{s-i}H_j(J)\right)}{D_k(J)} + \frac{D_{s+1}(J)}{D_k(J) \cdot H_0} > \frac{\displaystyle\sum_{i=k+1}^s\left( D_i(I)\displaystyle\prod_{j=1}^{s-i}H_j(I)\right)}{D_k(I)}.$$

Note that the above expression holds even if $\frac{D_i(J)}{D_k(J)}=\frac{D_i(I)}{D_k(I)}$, so the following equations are also satisfied for the base case. By rearranging terms, we have that $$\frac{D_k(I)}{\displaystyle\sum_{i=k+1}^s \left(D_i(I) \cdot \displaystyle\prod_{j=1}^{s-i}H_j(I) \right)}> \frac{D_k(J) \cdot H_0(J)}{H_0(J) \cdot \displaystyle\sum_{i=k+1}^s \left(D_i(J) \cdot \displaystyle\prod_{j=1}^{s-i}H_j(J) \right) + D_{s+1}(J)}.$$

By multiplying both sides of the relation by $$\frac{\displaystyle\prod_{i=1}^{s-k}H_i(I)}{\displaystyle\prod_{i=1}^{v+1}H'_{i+1}(I)} = \frac{\displaystyle\prod_{i=1}^{s-k}H_i(J)}{\displaystyle\prod_{i=1}^{v+1}H'_{i+1}(J)}$$ and adding $\frac{1}{H'_2(I)} = \frac{1}{H'_2(J)}$ to both sides, we generate the expressions for $C_{k, k+1}(I)$ and $C_{k, k+1}(J)$. Hence, we conclude $C_{k, k+1}(I) > C_{k, k+1}(J)$. This completes the induction and proves the proposition. 
\end{proof}

Corollary 3.6 of \cite{jm19} introduces an example of descent sets which reach the minimum ratio of $C_{a, b}$, but does not fully characterize all descent sets which obtain this ratio. One might wonder if there are any other descent sets not covered in Proposition \ref{prop:strong-3.6} (a strengthened version of the aforementioned Corollary 3.6), but there are in fact no others.  Proposition \ref{prop:add-row} finds descent sets with smaller ratios than a given set for $b>w(I)$, and we can use this to prove that the minimum of $C_{a, b}$ is never reached when $b>w(I)$. Theorem \ref{thm:num-min} expands on this concept, finding all $a$ and $b$ for fixed $I$ such that $C_{a, b}(I) = \frac{a!}{b!}$.

\begin{theorem} \label{thm:num-min}
Let $I$ be a non-empty descent set of positive integers. Let $s = (\lambda^I)_1-1$ and $(C_0, C_1, \ldots, C_s)$ be the Naruse-Newton coefficients of $I$, and let $w=w(I)$ where $w(I)$ is defined as above. Then,  $$\frac{C_0}{0!} = \frac{C_1}{1!} = \cdots = \frac{C_w}{w!}>\frac{C_{w+1}}{(w+1)!}>\cdots>\frac{C_s}{s!}.$$
\end{theorem}

\begin{proof} If $I$ is a shallow descent set, then $w \ge s$, and our result follows from Proposition \ref{prop:strong-3.6}.

For the remainder of this proof, assume $I$ is a deep descent set. It is evident by Propositions \ref{prop:jm19-2.4} and \ref{prop:strong-3.6} that $$\frac{C_0}{0!} = \frac{C_1}{1!} = \cdots = \frac{C_w}{w!} \ge \frac{C_{w+1}}{(w+1)!}\ge \cdots \ge\frac{C_s}{s!}.$$
Fix nonnegative integers $a$ and $b$ such that $s \ge b>w$ and $b>a \ge 0$. It is sufficient to prove that there does not exist a deep descent $I$ such that $$C_{a, b}(I) = \frac{a!}{b!}.$$

We claim that for all deep descent sets $I$, there exists another deep descent set $J$ such that $C_{a, b}(J) < C_{a, b}(I)$. This clearly suffices, as then there does not exist a deep descent set with a minimum ratio.

This claim is evident by Proposition \ref{prop:add-row}: for each deep descent set $I$, another deep descent set $J$ with a lower ratio can be formed by taking the descent set corresponding to the ribbon formed by appending a column of height $\lambda'_1$ to the left of $\mathbb{D}_{\operatorname{rib}}(\lambda^I)$. Hence, our claim, and our theorem, are proven.
\end{proof}

Furthermore, we can extend upon Theorem \ref{thm:num-min} to find a formal characterization of all descent sets $I$ which satisfy $C_{a, b}(I) = \frac{a!}{b!}$.

\begin{corollary} \label{cor:min-char}
Let $I$ be a non-empty descent set of positive integers. Let $s = (\lambda^I)_1-1$ and $(C_0, C_1, \ldots, C_s)$ be the Naruse-Newton coefficients of $I$, and let $w(I)$ be defined as above. Let positive integers $a$ and $b$ be defined such that $s \ge b > a \ge 0$. Then, $C_{a, b}(I) = \frac{a!}{b!}$ if and only if $w(I) \ge b$.
\end{corollary}

\begin{proof}
The direct assertion follows from Theorem \ref{thm:num-min}. The converse is also quite evident:  Assume $b>w(I)$. If $I$ is a shallow descent set, then $b>s$, for which there cannot exist a $b$. For the remainder of this proof, assume $I$ is a deep descent set. Similar to the proof of Theorem \ref{thm:num-min}, we can apply Proposition \ref{prop:add-row} to conclude. In particular, there must exist a descent set $J$ for each deep descent set $I$ with $w(I)<b$, such that $C_{a, b}(J)<C_{a, b}(I)$. Thus, each deep descent set $I$ with $w(I)<b$ cannot attain the minimum ratio, which is as desired.
\end{proof}

Using our result from Theorem \ref{thm:num-min}, we can also find all descent sets $I$ such that $C_{a, b}(I)$ for all $a$ and $b$ such that $s\ge b > a \ge 0$. This proves a major distinction between shallow descent sets and deep descent sets.

\begin{corollary} \label{cor:shallow}
Let $I$ be a non-empty set of positive integers, and define $\lambda = \lambda^I$ where $\lambda^I$ is defined in Section 2. Let $s = \lambda_1-1$ and $(C_0, C_1, \ldots, C_s)$ be the Naruse-Newton coefficients of $I$. The equality $\frac{C_0}{0!} = \frac{C_1}{1!} = \cdots = \frac{C_s}{s!}$ is satisfied if and only if $I$ is a shallow descent set.
\end{corollary}

\begin{proof}
Theorem \ref{thm:num-min} shows there do not exist any deep descent sets that satisfy $\frac{C_0}{0!} = \frac{C_1}{1!} = \cdots = \frac{C_s}{s!}$, since $w(I)<s$. Hence, for shallow descent sets, it holds that $w(I) \ge s$, and our result follows from Proposition \ref{prop:strong-3.6}.
\end{proof}

\bigskip
\section{Sets of All $C_{a, b}$}

For all integers $a, b$ such that $0 \le a < b$, define the set $$R_{a, b} = \left\{C_{a, b}(I) \mid I \textrm{ descent set and } b \le s\right\}.$$

In this section, we derive some properties of $R_{a, b}$, most importantly that it is unbounded above. Corollary \ref{cor:closure} may also be useful in finding the closure of $R_{a, b}$.

We now seek to find properties of $R_{a, b}$. First, the extrema: it is already shown in Proposition \ref{prop:jm19-2.4} that $\mathrm{min}(R_{a, b}) = \frac{a!}{b!}$, and sets which attain this minimum value are characterized in Corollary \ref{cor:min-char}. We now show that $R_{a, b}$ is unbounded above.

For a descent set $I$, define $\phi(I) = \{i+1\, | \, i \in I\} \cup \{1\}$ and $\psi(I) = \{i-g\, |\, i>g \textrm{ and } i \in I\}$, where $g$ is the least positive integer such that $g \not\in I$. In other terms, the function $\phi$ appends a cell to the bottom of the leftmost column of  $\mathbb{D}_{\operatorname{rib}}(\lambda^I)$ (as defined in Section 2), while the function $\psi$ entirely removes the leftmost column of $\mathbb{D}_{\operatorname{rib}}(\lambda^I)$.

\begin{theorem} \label{thm:no-max}
Let $I$ be a non-empty descent set of positive integers with $n = \max(I)+1$, and define $\lambda = \lambda^I$ where $\lambda^I$ is defined in Section 2. Let $s = \lambda_1-1$ and $(C_0, C_1, \ldots, C_s)$ be the Naruse-Newton coefficients of $I$. Let $C_{a, b}(I) = \frac{C_a(I)}{C_b(I)}$ for a set $I$, and let $\phi$ and $\psi$ be defined as above. Define $C_{s+1}(I) = 0$. Take nonnegative integers $a$ and $b$ such that $b > a \ge 0$, and either $s=b$ if $\lambda_1=2$ or $s > b$ if $\lambda_1>2$. Then, $$\displaystyle\lim_{n \rightarrow \infty} C_{a, b}(\phi^n(I))= \begin{cases} \infty, & \lambda_1 = 2, \\ C_{a, b}(\psi(I)), & \lambda_1>2. \end{cases}$$
\end{theorem}

\begin{proof}
Take $\lambda_1 = 2$ and variable $n$. The only $a$ and $b$ which satisfy $s = b > a \ge 0$ are $b=1$ and $a=0$, so we only must examine the ratio $C_{0, 1}$. Observe that $C_0$ is a polynomial in $n$, and its degree is equal to the maximum number of cells of the form $c_{k, 1}$ over all excited diagrams of $\mathbb{D}_{\operatorname{rib}}(\lambda^I)$ which contain $c_{1, 1}$. This value is equal to $\lambda'_2$, as seen in the initial excited diagram. Meanwhile, $C_1$ does not involve $n$, as the only excited diagram of $\mathbb{D}_{\operatorname{rib}}(\lambda^I)$ with no cells of the form $c_{1, k}$ also does not contain any cell of the form $c_{k, 1}$. Hence, for $\lambda_1=2$, $$\displaystyle\lim_{n \rightarrow \infty}C_{a, b}(\phi^n(I)) = \infty.$$

Next, assume $\lambda_1>2$. Note that for all $s>i \ge 0$, the coefficient $C_i$ must be a polynomial in $n$ of degree $\lambda'_2-2$. This is because the maximum number of cells in the form $c_{k, 1}$ over all excited diagrams of $\mathbb{D}_{\operatorname{rib}}(\lambda^I)$ with $i$ cells in their uppermost row must be equal to $\lambda'_2-2$. Thus, it is only necessary to examine the term of degree $\lambda'_2-2$ in the polynomial representations of $C_a$ and $C_b$. However, there exists a bijection between the set of  excited diagrams with fixed leftmost column and the set of excited diagrams of $\mathbb{D}_{\operatorname{rib}}(\lambda^I)$ without its leftmost column. Hence, for all $0 \le a \le s$, $$\displaystyle\lim_{n \rightarrow \infty} C_a(\phi^n(I)) = C_a(\psi(I)) \cdot \displaystyle\prod_{k=2}^{\lambda'_1} h_{(\phi^n(I))}(c_{k, 1}).$$
This directly results in the desired ratio for $\lambda_1>2$.
\end{proof}

% PROVE IS STRICTLY INCREASING
This theorem not only implies that $R_{a, b}$ is unbounded above, but that for every point $p \in R_{a, b}$ there exists a sequence of pairwise distinct points in $R_{a, b}$ converging to $p$. The natural next step to the convergence property of this construction is determining the closure of $R_{a, b}$. For any set $S \subseteq \mathbb{R}$, we define $\overline{S}$ to be the closure in the Euclidean topology of a set $S$. Using the construction from Theorem \ref{thm:no-max}, we now derive a property of $\overline{R_{a, b}}$.

\begin{corollary} \label{cor:closure}
For any integers $a$ and $b$ such that $a<b$ and subset $R' \subseteq R_{a, b}$ such that $|R_{a, b}-R'|$ is finite, we have that $\overline{R'} = \overline{R_{a, b}}$.
\end{corollary}

\begin{proof}
By the construction in Theorem \ref{thm:no-max} in the $\lambda_1>2$ case, every point $p$ in $R_{a, b}$ except $\frac{a!}{b!}$ is the limit of an increasing convergent sequence of pairwise distinct points $(p_1, p_2, \ldots)$ in $R_{a, b}-\{p\}$. 

The only outlier is the point $\frac{a!}{b!}$, as it is the minimum of $R_{a, b}$. However, there exists a decreasing convergent sequence of pairwise distinct points $(p_1, p_2, \ldots)$ that approaches $\frac{a!}{b!}$, namely $C_{a, b}(\{x, x+1\})$ as $x$ approaches infinity for arbitrarily large $x$.

Thus, only finitely many points are removed from each convergent sequence, hence every limit point of $R_{a, b}$ is also a limit point of $R'$. Since the closure is the set of all limit points, the two closures must coincide.
\end{proof}

For all integers $a \ge 0$, define the set $$T_a = \left\{C_{s-a-1, s-a}(I) \mid I \textrm{ descent set and } s=s(I) \ge a+1\right\}.$$Although $T_a$ is distinct from $R_a$, it remains similar in aspect of being a set of ratios of Naruse-Newton coefficients, only differing by the sequence in which the ratios are examined. Furthermore, our Theorem \ref{thm:no-max} allows us to generate a construction linking the closures of $T_a$ and $T_{a+1}$, thereby enabling us to find the closure of all $T_a$.

\begin{theorem} \label{thm:behind-closure}
For $a \ge 0$, it holds that $\overline{T_a} = \mathbb{R}_{\ge 0}$ and $T_0 = \mathbb{Q}_{>0}$.
\end{theorem}

\begin{proof}
We first prove that $\overline{T_0} = \mathbb{R}_{\ge 0}$. For all $a>2$ and $s>1$, construct set $I$ such that $\lambda^I = [s+1, s+1, s+1, 1, 1, \ldots, 1]$, where there are $(a-2)s+1$ parts. Then, it can be computed that $C_{s-1, s} (I) = \frac{a}{s+1}$. 

Thus, for any positive rational number $r = \frac{m}{n}$, the construction $a=3m$ and $s = 3n-1$ yields $C_{s-1, s}(I) = r$. Since all Naruse-Newton coefficients must be positive integers, we have that $T_0 = \mathbb{Q}_{>0}$. Of course, each $r\in \mathbb{R}_{>0}\backslash \mathbb{Q}_{>0}$ is a limit point of $R_0$, as the limit of an infinitely increasing convergent sequence of points $(\frac{\lfloor 2^0 r\rfloor}{2^0}, \frac{\lfloor 2^1 r\rfloor}{2^1}, \frac{\lfloor 2^2 r\rfloor}{2^2}, \ldots)$. Additionally, $0$ is the limit of the decreasing convergent sequence of points $(1, \frac{1}{2}, \frac{1}{3}, \ldots)$. Hence, we have that $\overline{T_0} = \mathbb{R}_{\ge 0}$.

We conclude by induction. Assume that for fixed $a \ge 0$, we know $\overline{T_a} = \mathbb{R}_{\ge 0}$. We will show that $\overline{T_{a+1}} = \mathbb{R}_{\ge 0}$. The base case, namely $a=0$, has already been proven. 

For any point $t \in T_a$, there exists a set $I'$ such that $t = C_{s-a-1, s-a}(I')$. Then, define the sequence of sets $(I_k)_{k=0}^\infty$ such that $I_k = [k] \cup \{i+k+1 \mid i \in I'\}$. By Theorem \ref{thm:no-max}, we have that $t$ is the limit of the increasing convergent sequence of points $(C_{s-a-2, s-a-1}(I_0), C_{s-a-2, s-a-1}(I_1), C_{s-a-2, s-a-1}(I_2), \ldots)$, each point in the sequence an element of $T_{a+1}$. Hence, $T_a \subseteq \overline{T_{a+1}}$. We thus have that $\overline{T_a} \subseteq \overline{\overline{T_{a+1}}}$. However, it is clear that $\overline{T_{a+1}} = \overline{\overline{T_{a+1}}}$, as $\overline{T_{a+1}}$ is already closed by the definition of a closure.

Finally, we have that $\mathbb{R}_{\ge 0} = \overline{T_a} \subseteq \overline{T_{a+1}}$. As all Naruse-Newton coefficients are positive integers, we have that $\min(T_{a+1})>0$, thus $\min(\overline{T_{a+1}}) \ge 0$. We conclude that $\overline{T_{a+1}} = \mathbb{R}_{\ge 0}$, as desired.
\end{proof}

\section{Doubletons}
This following section concerns coefficients $C_i$ of descent set $I$ when $|I| = 2$. In particular, we compare each $C_i$ with $C_{i+1}$, and find all $I$ such that the sequence of Naruse-Newton coefficients of $I$ is unimodal or log-concave. These results can motivate and provide context for properties of Naruse-Newton coefficients of $I$ with $|I|>2$.

We first begin by comparing each $C_i$ to the following $C_{i+1}$.

\begin{theorem} \label{thm:two-ineq}
Let $I = \{a, b\}$ for positive integers $a<b$. Let $(C_0, C_1, \ldots, C_{b-2})$ be the Naruse-Newton coefficients of $I$. Then, $$C_0 \le C_1 < C_2 < \cdots < C_{b-a-1}>C_{b-a}< C_{b-a+1}< \cdots \le C_{b-2}.$$ The equality case with $C_0=C_1$ occurs if and only if $b>a+1$. The equality case with $C_{b-3}=C_{b-2}$ occurs if and only if $I = \{3, 4\}$.
\end{theorem}

\begin{proof}
By Proposition \ref{prop:strong-3.6}, we have 
$$\frac{C_0}{0!} = \frac{C_1}{1!} = \cdots = \frac{C_{b-a-1}}{(b-a-1)!}.$$

Hence, the equality case with $C_0 = C_1$ occurs when $b-a-1>0$, as stated. Let $k = b-a$. We will now show that $$C_{k}<C_{k+1}<\cdots \le C_{b-2}.$$ We proceed by induction on $a$, with constant $k$. The base case occurs when $I = \{2, k+2\}$. Then, we have that $b-a=b-2=k$, so the base case satisfies the inequality. Assume the desired inequality holds true for $I = \{b-k, b\}$. We will prove it holds true for $J = \{b-k+1, b+1\}$. By Proposition \ref{prop:add-row}, we know $C_{i, j}(J)<C_{i, j}(I)$ for $0 \le i<j$. Thus, if $C_i(I) \le C_j(I)$, then $C_i(J)<C_j(J)$. Hence, by our assumption, we have $$C_k(J)<C_{k+1}(J)<\cdots<C_{b-2}(J).$$

We will now show that $C_{b-2}(J) \le C_{b-1}(J)$. We compute $$C_{b-2}(J) = \frac{(b-k-1)!(b-1)!(2b-k+1)}{k}.$$

Furthermore, we compute that $$C_{b-1}(J) = \frac{(b-k)!b!}{k}.$$

By dividing, we now have $$C_{b-2, b-1}(J) = \frac{2b-k+1}{(b-k)b}.$$

Note that $\frac{2b-k+1}{(b-k)b}\le 1$ when $a=b-k+1>1$, hence we have as desired, and $C_{b-2}(J)\le C_{b-1}(J)$. The only equality case occurs when $k=1$ and $b=4$. Since the base case has $a=2$, we know that all cases can be produced from the base case.

We now prove that $C_{k-1}>C_k$. We compute that $$C_{k-1} = (k-1)! \cdot \left(\displaystyle\sum_{i=0}^{a-1} \left(\displaystyle\prod_{j=1}^i j \displaystyle\prod_{\ell=k+2+i}^b \ell\right)\right).$$

We also compute that $$C_{k} = (k-1)!(k+1)\left(\displaystyle\sum_{i=1}^{a-1}\left(\displaystyle\prod_{j=1}^i j \displaystyle\prod_{\ell=k+2+i}^b \ell\right)\right).$$

Finally, it is easily computed that $C_{k-1}>C_k$. We conclude that $C_0 \le C_1<C_2<\cdots <C_{k-1}>C_k< C_{k+1}<\cdots \le C_{b-2}$. This is as desired.
\end{proof}

A \emph{unimodal} sequence is a sequence $(a_k)^m_{k=0}$ such that there exists an $i$ in the range $1 \le i \le m$ with $a_0 \le a_1 \le \cdots \le a_i$ and $a_i \ge a_{i+1} \ge \cdots \ge a_m$. Using the inequalities of our Theorem \ref{thm:two-ineq}, we find all $I$ of size $2$ such that the Naruse-Newton coefficients of $I$ form a unimodal sequence. 

\begin{corollary} \label{cor:two-uni}
Let $I = \{a, b\}$ for positive integers $a<b$. The sequence of Naruse-Newton coefficients is unimodal if and only if $a=1$, $a=2$, or $I = \{3, 4\}$.
\end{corollary}

\begin{proof}
If the sequence is unimodal, there cannot exist an element $C_i$ such that $C_{i-1}>C_i<C_{i+1}$. However, by Theorem \ref{thm:two-ineq}, we know $C_{b-a-1}>C_{b-a}<C_{b-a+1}< \cdots \le C_{b-2}$. Thus, if the sequence of Naruse-Newton coefficients of $\{a, b\}$ is unimodal, then $b-a+1 \ge b-2$, and $a \le 3$. If $a=1$, then $C_0 \le C_1 < \cdots <C_{b-2}$, so the sequence must be unimodal for all $b$. If $a=2$, then $C_0 =C_1 < \cdots <C_{b-3}>C_{b-2}$, so the sequence must be unimodal for all $b$. If $a=3$, then $C_0=C_1 < \cdots < C_{b-4}>C_{b-3} \le C_{b-2}$. This sequence can only be unimodal when $C_{b-3} = C_{b-2}$, so the only case the sequence is unimodal for $a=3$ is the $I = \{3, 4\}$ case. We have as desired.
\end{proof}

This corollary may give insight into finding unimodal sequences of Naruse-Newton coefficients for $|I| \ge 3$. We also obtain a nice result by finding the probability that a randomly chosen sequence of Naruse-Newton coefficients for a doubleton $I$ is unimodal.

\begin{corollary} \label{cor:two-prob}
Let $n \ge 4$ be a positive integer. Let $I$ be chosen uniformly at random from ${[n] \choose 2}$. Then, the probability that the sequence of Naruse-Newton coefficients of $I$ is unimodal is exactly $\frac{4}{n}$.
\end{corollary}

In particular, this probability approaches $0$ as $n$ approaches $\infty$. 

A \emph{log-concave} sequence $(a_k)^m_{k=0}$ is defined such that for all $0 < k < m$,  we have $a_k^2 \ge a_{k-1}a_{k+1}$. Because all log-concave sequences are unimodal, we can use Corollary \ref{cor:two-uni} to determine all doubletons $I$ (of which there are four) such that the sequence of Naruse-Newton coefficients of $I$ is log-concave.

\begin{corollary} \label{cor:two-log}
Let $I = \{a, b\}$ for positive integers $a<b$. The sequence of Naruse-Newton coefficients of $I$ is log-concave if and only if $I = \{1, 2\}$, $I = \{1, 3\}$, $I = \{2, 3\}$, or $I = \{2, 4\}$.
\end{corollary}

\begin{proof}
All log-concave sequences must be unimodal, so we narrow our search down to unimodal sequences of Naruse-Newton coefficients. Hence, we conduct casework on $I$ as per the cases listed in Corollary \ref{cor:two-uni}. If $a=1$ and $b \ge 4$, then $C_0 = C_1 = 1$ and $C_2 = 2$, so the sequence is not log-concave. Thus, the only sequences with $a=1$ that are log-concave correspond to $I = \{1, 2\}$ and $I = \{1, 3\}$. If $a=2$ and $b \ge 5$, then $C_0 = C_1 = b+1$ and $C_2 = 2b+2$, so the sequence is not log-concave. Thus, the only sequences with $a=2$ that are log-concave correspond to $I = \{2, 3\}$ and $I = \{2, 4\}$. Finally, if $I = \{3, 4\}$, then $C_0 = 18$ and $C_1 = C_2 = 12$, so the sequence is not log-concave.
\end{proof}

\section{Ribbons of Staircase Shape}

A \emph{ribbon of staircase shape} is defined as a ribbon that corresponds to a descent set $I = \{1, 3, \ldots, 2k+1\}$ for a positive integer $k$. In this section, we examine ribbons of staircase shape and properties of their Naruse-Newton coefficients.

\begin{proposition} \label{prop:staircase}
Let $I = \{1, 3, \ldots, 2k+1\}$ for positive integer $k$. Let $(C_0, C_1, \ldots, C_k)$ be the Naruse-Newton coefficients of $I$. Then, $$C_0=C_1>C_2> \cdots > C_k.$$
\end{proposition}

\begin{proof}
We use induction on $k$. The base case occurs when $k=1$, and satisfies the desired condition as $C_0 = C_1 = 1$.

Assume the proposition is true for $J = \{1, 3, \ldots, 2k-1\}$. We shall prove it holds true for $I = \{1, 3, \ldots, 2k+1\}$. By Proposition \ref{prop:strong-3.6}, we have $C_0(I) = C_1(I)$. Then, we express each $C_a(I)$ as a sum of multiples of $C_b(J)$, by casework on the number of unexcited cells in the second row of $\mathbb{D}_{\operatorname{rib}}(\lambda^I)$. Thus, we have $$C_i(I) = \left(\displaystyle\prod_{j=1}^{i} (2j-1) \right) \displaystyle\sum_{\ell=i-1}^{k-1} \left(C_{\ell}(J) \displaystyle\prod_{m=\ell+2}^k (2m+1) \right).$$

We now compare $C_i(I)$ and $C_{i+1}(I)$ for $0<i \le k-1$. Note that $$C_i(I) = \frac{C_{i-1}(J) \displaystyle\prod_{j=1}^{k} (2j+1) + C_{i+1}(I)}{2i+1}.$$

However, since $C_{i-1}(J)\ge C_i(J)> \cdots > C_{k-1}(J)$ by our assumption, it is clear by summation of coefficients that $$C_{i-1}(J) \displaystyle\prod_{j=1}^{k}(2j+1) > 2i \cdot C_{i+1}(I).$$

Hence, we have $C_i(I)>C_{i+1}(I)$ for $0<i \le k-1$, as desired.
\end{proof}

To further understand the properties Naruse-Newton coefficients corresponding to ribbons of staircase shape, we examine ratios between them. In particular, we characterize the number triangle formed by $(C_{a}(\{1, \ldots, 2b+1\}))_{b \ge a \ge 0}$. In this triangle, the first row occurs when $b=0$, and contains only $C_0 = 0$. The second row occurs when $b=1$, and contains $C_0 = 1$ and $C_1 = 1$. The third row contains $6, 6, 3$, the fourth row contains $255, 255, 135, 45$, and the fifth row contains $97650, 97650, 51975, 18900, 4725$. For a positive integer $x$, the double factorial is defined as $x!! = \displaystyle\prod_{i=0}^{\lceil \frac{x}{2} \rceil-1} (x-2i)$. We now find a computational formula for each element of the number triangle, based upon elements in prior rows.

\begin{proposition}\label{prop:recurrence}
Let $C_k(I)$ represent the $(k+1)$-th Naruse-Newton coefficient of a non-empty set of positive integers $I$. Then, the number triangle formed by $C_{a}(\{1, \ldots, 2b+1\})_{b \ge a \ge 0}$ is computed using:
\begin{enumerate}
    \item $C_0(\{1\}) = 1$.
    \item For all $b \ge a \ge 1$, we have $$C_a(\{1, \ldots, 2b+1\}) = (2a-1)!! \displaystyle\sum_{i=a-1}^{b-1} \left(C_{i}(\{1, \ldots, 2b-1\}) \cdot \frac{(2b+1)!!}{(2i+3)!!}\right).$$
    \item For all $b \ge 1$, we have $C_{0}(\{1, \ldots, 2b+1\}) =C_{1}(\{1, \ldots, 2b+1\})$.
\end{enumerate}
\end{proposition} 

\begin{proof}
We prove the three parts of this proposition sequentially, beginning with the first. Note that $I = \{1\}$ corresponds to the ribbon $\mathbb{D}_{\operatorname{rib}}(\lambda^I)$ where $\lambda^I = (2, 2, 1)$. In this case, we know $C_0 = 1$, as desired.

The second part is proven by separating the set of excited diagrams with $b-a$ cells in the first row by the number of cells in the second row, similar to our process in Proposition \ref{prop:add-row}. In $C_a$, the coefficient of $C_{i}$ for $a-1 \le i \le b-1$ must thus be $(2a-1)!! \cdot \frac{(2b+1)!!}{(2i+3)!!}$, the former term from the product of the hook lengths of the rightmost cells of the second row, the latter term from the product of the hook lengths of the leftmost cells of the second row. Hence, by taking the summation, we arrive at our desired formula.

Finally, the third part is proven directly from Corollary \ref{cor:min-char}, as $w(\{1, \ldots, 2b+1\}) = 1$. Evidently, these three equations define the number triangle, as the second and third equation generate each subsequent $(C_i(\{1, \ldots, 2b+1\}))_{i=0}^b$ given $(C_i(\{1, \ldots, 2b-1\}))_{i=0}^{b-1}$. Meanwhile, the first equation generates the ``base case", that is, $C_0(\{1\}) = 1$.
\end{proof}

Our number triangle now allows us to prove a key result on the polynomiality of a particular ratio of Naruse-Newton coefficients of ribbons of staircase shape, namely Theorem \ref{thm:tripoly}. While this theorem does not hold for most other ribbons, it holds for ribbons of staircase shape. Moreover, it is quite rare to be able to generate a polynomial equal to a ratio of sums of products, such a property representing the beauty of ribbons of staircase shape.

We first prove an important and quite beautiful lemma used in the proof of Theorem \ref{thm:tripoly} to show the existence of a set of polynomial coefficients.

\begin{lemma} Define the $(k+1)\times(k+1)$ matrix $\widetilde{A}_k(x)$ as:
$$\widetilde{A}_k(x) = \begin{bmatrix} 1 & x & x^2 & \cdots & x^{k-2} & x^{k-1} & x^k \\ 0 & 0 & 0 &\cdots & 0 & -2 & -2{k \choose 2} + {k \choose 1} \\ 0 & 0 & 0 & \cdots & -4 & -2{k-1 \choose 2} + {k-1 \choose 1} & 2{k \choose 3} - {k \choose 2} \\ 0 & 0 & 0 & \cdots & -2{k-2 \choose 2} + {k-2 \choose 1} & 2 {k-1 \choose 3}-{k-1 \choose 2} & -2{k \choose 4}+{k\choose 3} \\ \vdots & \vdots & \vdots & \ddots & \vdots &\vdots & \vdots \\ -2k & 1 & -1 & \cdots & (-1)^{k-1} & (-1)^k & (-1)^{k+1} \end{bmatrix}.$$

Then, $\det(\widetilde{A}_k(x)) = (-1)^{(k^2-k)/2} \cdot k! \cdot \displaystyle\prod_{i=1}^k (2i-3-2x)$.  \label{lem:matrix}
\end{lemma}

\begin{proof}
Let $[A]_i$ denote the $i$th row of a matrix $A$, and let $[A]_{i, j}$ denote the $j$th element of the $i$th row of $A$.

Note that $\det(\widetilde{A}_k(x))$ must be a polynomial with degree at most $k$, as all terms of positive degree are located on the row of the matrix. We will first show that the roots of $\det(\widetilde{A}_k(x))$ are $-\frac{1}{2}, \frac{1}{2}, \ldots, \frac{2k-3}{2}$. 

We now show that $-\frac{1}{2}$ is a root of $\det(\widetilde{A}_k(x))$. We perform a left multiplication on $\widetilde{A}_k(x)$ to generate $\widetilde{B}_k(x)$ as follows:
\begin{equation}\widetilde{B}_k(x) = \begin{bmatrix} 1 & 0 & 0  & \cdots & 0 & 0 & 0 \\ 0 & 1 & 0 &  \cdots & 0 & 0 & 0 \\ 0 & 0 & 1 &  \cdots & 0 & 0 & 0 \\ \vdots & \vdots & \vdots &  \ddots & \vdots & \vdots & \vdots \\ 2k & \left(-\frac{1}{2}\right)^{k-1} & \left(-\frac{1}{2}\right)^{k-2} &  \cdots & \frac{1}{4} & -\frac{1}{2} & 1 \end{bmatrix} \cdot \widetilde{A}_k(x). \label{eq:bmatrix}\end{equation}

However, $[\widetilde{B}_k(x)]_i = [\widetilde{A}_k(x)]_i$ for $1\le i \le k$. By expanding Equation \eqref{eq:bmatrix}, we have that $$[\widetilde{B}_k(x)]_{k+1} = \begin{bmatrix}0& 2kx+k &2kx^2 -k/2 & \cdots & 2kx^k+(-2k)\left(-\frac{1}{2}\right)^k \end{bmatrix}.$$ 

Hence, the elements of $[\widetilde{B}_k(-\frac{1}{2})]_{k+1}$ must all be zero, so $\det(\widetilde{B}_k(-\frac{1}{2})) = 0$. Thus, Equation \eqref{eq:bmatrix} gives $\det(\widetilde{A}_k(-\frac{1}{2})) = 0$, and $-\frac{1}{2}$ must be a root of $\det(\widetilde{A}_k(x))$.

Assume $\rho=a-\frac{1}{2}$ is a root of $\det(\widetilde{A}_k(x))$ for $0 \le a < k-1$. We will show that $\rho+1$ is a root. Completing this induction proves that $-\frac{1}{2}, \frac{1}{2}, \ldots, \frac{2k-3}{2}$ are the $k$ roots of $\det(\widetilde{A}_k(x))$, as we have already proven the base case of $a=0$.

We first perform a left multiplication on $\widetilde{A}_k(x)$ to produce the matrix $\widetilde{A}'_k(x)$.
\begin{equation} \widetilde{A}_k'(x) = \begin{bmatrix} (2k-3-2\rho)(\rho+1)^{-k+1} & 1 & (\rho+1)^{-1} & \cdots & (\rho+1)^{-k+1} \\ 0 & 1 & 0 & \cdots & 0 \\ 0 & 0 & 1 & \cdots & 0 \\ \vdots & \vdots & \vdots & \ddots & \vdots \\ 0 & 0 & 0& \cdots & 1 \end{bmatrix} \cdot \widetilde{A}_k(x). \label{eq:inmatrix}\end{equation}

We now expand Equation \eqref{eq:inmatrix} to determine the first row of our new matrix, such that for $0 \le i \le k$,  $$[\widetilde{A}_k'(x)]_{1, i+1} =-(\rho+1)^{-k+1}(\rho+2)\rho^i+(2\rho-2k+3)(\rho+1)^{-k+1}((\rho+1)^i-x^i).$$

Let $d = -(\rho+1)^{-k+1}(\rho+2)$. Hence, for $0 \le i \le k$, $$[\widetilde{A}'_k(x)]_{1, i+1} = -(\rho+1)^{-k+1}(\rho+2)\rho^i = d\cdot \rho^i.$$ 

However, Equation \eqref{eq:inmatrix} implies that $[\widetilde{A}'_k(x)]_i = [\widetilde{A}_k(x)]_i$ for $2\le i \le k+1$, thus we have $\widetilde{A}'_k(\rho+1) = \operatorname{diag}(d, 1, 1, \ldots, 1) \cdot \widetilde{A}_k(\rho)$, where $\operatorname{diag}(d, 1, 1, \ldots, 1)$ is defined as the $(k+1) \times (k+1)$ diagonal matrix with elements $(d, 1, 1, \ldots, 1)$.

By taking the determinants of this equation, we conclude that
$$\det(\widetilde{A}'_k(\rho+1)) = d \cdot \det(\widetilde{A}_k(\rho)) = -((\rho+1)^{-k+1})((\rho+1)+1) \det(\widetilde{A}_k(\rho)) = 0.$$

Plugging this into Equation \eqref{eq:inmatrix}, we have that $$\det(\widetilde{A}_k(\rho+1)) \cdot (2\rho+3-2k)(\rho+1)^{-k+1} = \det(\widetilde{A}'_k(\rho+1)) = 0.$$

Thus, because $\rho \neq \frac{2k-3}{2}$ and $\rho \neq -1$, we have that $\det(\widetilde{A}_k(\rho+1))=0$, in other words, $\rho+1$ is a root of $\widetilde{A}_k(x)$. Note that this argument fails when $\rho = \frac{2k-3}{2}$, limiting our induction.

Finally, we have found the $k$ roots of $\det(\widetilde{A}_k(x))$. We now express $\det(\widetilde{A}_k(x)) = a \cdot \displaystyle\prod_{i=1}^k (x-i+\frac{3}{2}).$ The leading coefficient of this polynomial, $a$, is obtained by computing the determinant of the matrix formed by rows $2$ to $k+1$ and columns $1$ to $k$ of $\widetilde{A}_k(x)$. Thus, we have that $$a  = (-1)^{(k^2+k)/2}(2k)!!.$$ 
Finally, we expand and conclude that $$\det(\widetilde{A}_k(x)) = (-1)^{(k^2-k)/2} \cdot k! \cdot \displaystyle\prod_{i=1}^k (2i-3-2x).$$
\end{proof}

We now use the aforementioned lemma to prove the polynomiality of a particular ratio of Naruse-Newton coefficients of ribbons of staircase shape.

\begin{theorem} \label{thm:tripoly}
Let $C_k(I)$ represent the $(k+1)$-th Naruse-Newton coefficient of a non-empty set of positive integers $I$, and $C_{j, k}(I) = \frac{C_j(I)}{C_k(I)}$. For every nonnegative integer $i$, there exists a unique polynomial $P_i(t) \in \mathbb{Q}[t]$ of degree at most $i$ such that $P_i(a) = C_{a-i, a}(\{1, 3, \ldots, 2a+1\})$.
\end{theorem}

\begin{proof}
We proceed by strong induction on $k$. The base case is $k=0$, with $P_0(a) = C_{a, a}(\{1, \ldots, 2a+1\}) = 1$. 

Fix $k$ to be a positive integer. Assume that there exist unique polynomials $P_0(t), P_1(t), \ldots, P_{k-1}(t) \in \mathbb{Q}[t]$ such that $P_i(a) = C_{a-i, a}(\{1, 3, \ldots, 2a+1\})$ for $0 \le i \le k-1$ and $P_i(a)$ is of degree at most $i$. Then, we will prove that there exists a polynomial $P_k(t) \in \mathbb{Q}[t]$ of degree at most $k$ such that $P_k(a) = C_{a-k, a}(\{1, 3, \ldots, 2a+1\})$.

Take the polynomial $Q_k(a) = \displaystyle\sum_{i=0}^k c_i a^i$ for unknown $c_i$, and let $P_{k-1}(a)= \displaystyle\sum_{i=0}^{k-1}d_i a^i$ for known rational $d_i$. We now show that $C_{a-k, a}(\{1, \ldots, 2a+1\})$ is polynomial in $a$ and has degree at most $k$ by proving the existence of such a polynomial. In other words, if we prove that there exists a unique rational sequence $(c_i)_{i=0}^k$ such that $Q_k(a) = C_{a-k, a}(\{1, \ldots, 2a+1\})$ and $\deg(Q_k(a)) \le k$, then our proof is complete. 

The polynomial $Q_k(t)$ is characterized by the two equations $Q_k(a) \cdot (2a-2k+1)-Q_k(a-1) \cdot (2a+1) = P_{k-1}(a)$ and $Q_k(k) = P_{k-1}(k)$, generated by substituting $P_{k-1}(t)$ and $Q_k(t)$ in the second and third equations of Proposition \ref{prop:recurrence}. Hence, these two equations must characterize the sequence $(c_i)_{i=0}^k$, and proving there exists a possible sequence with nonzero $c_k$ that satisfies both equations finishes our proof.

The first of these two aforementioned equations is equivalent to $$(c_ka^k+\cdots+c_0)(2a-2k+1)-(c_k(a-1)^k+\cdots+c_0)(2a+1) = d_{k-1}a^{k-1}+\cdots+d_0.$$

The second of these is equivalent to $$k^kc_k+\cdots+c_0 = d_{k-1}k^{k-1}+\cdots+d_0.$$

By equating the coefficients of terms of degree $i$ for $0 \le i < k$ in the first equation, we are able to create $k+1$ equations regarding $(c_i)_{i=0}^k$. We define $A_k$ to be the $(k+1)\times (k+1)$ matrix such that $$A_k\begin{bmatrix} c_0 \\ c_1 \\ \vdots \\ c_k \end{bmatrix} = \begin{bmatrix} d_{k-1}k^{k-1}+\cdots+d_0 \\ d_{k-1} \\ \vdots \\ d_0 \end{bmatrix}.$$

By expanding the aforementioned equations, we generate the terms of $A_k$: 
$$A_k = \begin{bmatrix} 1 & k & k^2 & \cdots & k^{k-2} & k^{k-1} & k^k \\ 0 & 0 & 0 &\cdots & 0 & -2 & -2{k \choose 2} + k \\ 0 & 0 & 0 & \cdots & -4 & -2{k-1 \choose 2} + k & 2{k \choose 3} - k/2\\ 0 & 0 & 0 & \cdots & -2{k-2 \choose 2} + k & 2 {k-1 \choose 3}-k/2 & -2{k \choose 4}+k/4 \\ \vdots & \vdots & \vdots & \ddots & \vdots &\vdots & \vdots    \\ -2k & k & -k/2 & \cdots & \frac{k}{(-2)^{k-3}} & \frac{k}{(-2)^{k-2}} & \frac{k}{(-2)^{k-1}} \end{bmatrix}.$$

Proving that there exists a unique $(c_i)_{i=0}^k$ is equivalent to proving $\det(A_k) \neq 0$.

We refer to our Lemma \ref{lem:matrix}, proven above, to complete our proof. Since $A_k = \widetilde{A}_k(k)$, so $\det(A_k) = (-1)^{(k^2-k)/2} \cdot k! \cdot \displaystyle\prod_{i=1}^k (2i-3-2k) \neq 0$. Thus, the sequence $(c_i)_{i=0}^k$ is determined uniquely. Furthermore, this implies that $\deg(Q_k(t)) \le k$, so our strong induction, and thereby our proof, is complete.
\end{proof}

\bigskip
\section*{Acknowledgements}
I thank Pakawut Jiradilok for his mentorship and guidance. I am very grateful for his many insights from his work and suggestions for this paper. I would also like to thank the MIT PRIMES-USA program and every individual involved with them for making this research possible, and for hosting and allowing me to present at the Tenth Annual Fall Term PRIMES Conference. In particular, I would like to thank the program director, Slava Gerovitch, the head mentor, Tanya Khovanova, the assistant head mentor, Alex Vitanov, and the chief research advisor, Pavel Etingof. 

\section*{Appendix A: Table of Naruse-Newton Coefficients for $I \subseteq [7]$}
In this appendix, we list computed values of Naruse-Newton coefficients for $127$ descent sets, namely the nonempty subsets of $\{1, 2, 3, 4, 5, 6, 7\}$. A Java program to compute the Naruse-Newton coefficients of a provided descent set can be found on my GitHub page, at \newline \url{https://github.com/andrewcai31/coefficient-calculator}.\vspace{0.3 cm}

\noindent \begin{longtable}{|p{1cm}|p{0.5cm}|p{1cm}|p{1cm}|p{1cm}|p{1cm}|p{1cm}|p{1cm}|p{1cm}|}
\hline
 \multicolumn{9}{|c|}{$|I| = 1$} \\
 \hline 
$I$ & $s$ & $C_0$ & $C_1$ & $C_2$ & $C_3$ & $C_4$ & $C_5$ & $C_6$ \\
\hline
\endfirsthead
 \hline 
$I$ & $s$ & $C_0$ & $C_1$ & $C_2$ & $C_3$ & $C_4$ & $C_5$ & $C_6$ \\
\hline
\endhead
\hline
\endfoot
\hline
\endlastfoot
$\{1\}$ & $0$ & $1$ & & & & & & \\
$\{2\}$ & $1$ & $1$ & $1$ & & & & & \\
$\{3\}$ & $2$ & $1$ & $1$ & $2$ & & & & \\
$\{4\}$ & $3$ & $1$ & $1$ & $2$  & $6$ & & & \\
$\{5\}$ & $4$ & $1$ & $1$ & $2$ & $6$ &$24$ & & \\
$\{6\}$ & $5$ & $1$ & $1$ & $2$ & $6$ &$24$ & $120$ & \\
$\{7\}$ & $6$ & $1$ & $1$ & $2$ & $6$ &$24$ & $120$ & $720$\\
\end{longtable}
\noindent \begin{longtable}{|p{1.4 cm}|p{0.5 cm}|p{1.17 cm}|p{1.17cm}|p{1.17cm}|p{1.17cm}|p{1.17cm}|p{1.17cm}|}
\hline
\multicolumn{8}{|c|}{$|I| = 2$} \\
\hline
$I$ & $s$ & $C_0$ & $C_1$ & $C_2$ & $C_3$ & $C_4$ & $C_5$ \\
\hline
\endfirsthead
\hline
$I$ & $s$ & $C_0$ & $C_1$ & $C_2$ & $C_3$ & $C_4$ & $C_5$ \\
\hline
\endhead
\hline
\endfoot
\hline
\endlastfoot
$\{1, 2\}$ & $0$ & $1$ & & & & &  \\
$\{1, 3\}$ & $1$ & $1$ & $1$ & & & & \\
$\{1, 4\}$ & $2$ & $1$ & $1$ & $2$& & & \\
$\{1, 5\}$ & $3$ & $1$ & $1$ & $2$& $6$& &  \\
$\{1, 6\}$ & $4$ & $1$ & $1$ & $2$& $6$&$24$ &  \\
$\{1, 7\}$ & $5$ & $1$ & $1$ & $2$& $6$&$24$ & $120$ \\
\hline
$\{2, 3\}$ & $1$ & $4$ & $2$ & & & & \\
$\{2, 4\}$ & $2$ & $5$ & $5$ & $3$& & & \\
$\{2, 5\}$ & $3$ & $6$ & $6$ & $12$& $8$& &  \\
$\{2, 6\}$ & $4$ & $7$ & $7$ & $14$& $42$&$30$ &  \\
$\{2, 7\}$ & $5$ & $8$ & $8$ & $16$& $48$&$192$ & $144$ \\
\hline
$\{3, 4\}$ & $2$ & $18$ & $12$ & $12$& & & \\
$\{3, 5\}$ & $3$ & $27$ & $27$ & $21$& $24$& &  \\
$\{3, 6\}$ & $4$ & $38$ & $38$ & $76$& $64$&$80$ &  \\
$\{3, 7\}$ & $5$ & $51$ & $51$ & $102$& $306$&$270$ & $360$ \\
\hline
$\{4, 5\}$ & $3$ & $96$ & $72$ & $96$& $144$& &  \\
$\{4, 6\}$ & $4$ & $168$ & $168$ & $144$& $216$&$360$ &  \\
$\{4, 7\}$ & $5$ & $272$ & $272$ & $544$& $496$&$800$ & $1440$ \\
\hline
$\{5, 6\}$ & $4$ & $600$ & $480$ & $720$& $1440$&$2880$ &  \\
$\{5, 7\}$ & $5$ & $1200$ & $1200$ & $1080$& $1800$&$3960$ & $8640$ \\
\hline
$\{6, 7\}$ & $5$ & $4320$ & $3600$ & $5760$& $12960$&$34560$ & $86400$ \\
\end{longtable}

\noindent \begin{longtable}{|p{1.4cm}|p{0.5cm}|p{1.5cm}|p{1.5cm}|p{1.5cm}|p{1.5cm}|p{1.5cm}|}
\hline
 \multicolumn{7}{|c|}{$|I| = 3$} \\
 \hline 
 $I$ & $s$ & $C_0$ & $C_1$ & $C_2$ & $C_3$ & $C_4$\\
\hline
\endfirsthead
 \hline 
$I$ & $s$ & $C_0$ & $C_1$ & $C_2$ & $C_3$ & $C_4$ \\
\hline
\endhead
\hline
\endfoot
\hline
\endlastfoot
$\{1, 2, 3\}$ & $0$ & $1$ & & & &  \\
$\{1, 2, 4\}$ & $1$ & $1$ & $1$& & &  \\
$\{1, 2, 5\}$ & $2$ & $1$ & $1$&$2$ & &  \\
$\{1, 2, 6\}$ & $3$ & $1$ & $1$&$2$ &$6$ &  \\
$\{1, 2, 7\}$ & $4$ & $1$ & $1$&$2$ &$6$ & $24$ \\
\hline
$\{1, 3, 4\}$ & $1$ & $5$ & $2$& & &  \\
$\{1, 3, 5\}$ & $2$ & $6$ & $6$&$3$ & &  \\
$\{1, 3, 6\}$ & $3$ & $7$ & $7$&$14$ &$8$ &  \\
$\{1, 3, 7\}$ & $4$ & $8$ & $8$&$16$ &$48$ & $30$ \\
\hline
$\{1, 4, 5\}$ & $2$ & $22$ & $14$&$12$ & &  \\
$\{1,4, 6\}$ & $3$ & $32$ & $32$&$24$ &$24$ &  \\
$\{1, 4, 7\}$ & $4$ & $44$ & $44$&$88$ &$72$ & $80$ \\
\hline
$\{1,5, 6\}$ & $3$ & $114$ & $84$&$108$ &$144$ &  \\
$\{1, 5, 7\}$ & $4$ & $195$ & $195$&$165$ &$240$ & $360$ \\
\hline
$\{1, 6, 7\}$ & $4$ & $696$ & $552$&$816$ &$1584$ & $2880$ \\
\hline
$\{2, 3, 4\}$ & $1$ & $18$ & $6$& & &  \\
$\{2, 3, 5\}$ & $2$ & $22$ & $22$&$8$ & &  \\
$\{2, 3, 6\}$ & $3$ & $26$ & $26$&$52$ &$20$ &  \\
$\{2, 3, 7\}$ & $4$ & $30$ & $30$&$60$ &$180$ & $72$ \\
$\{2, 4, 5\}$ & $2$ & $128$ & $56$&$24$ & &  \\
$\{2, 4, 6\}$ & $3$ & $183$ & $183$&$99$ &$45$ &  \\
$\{2, 4, 7\}$ & $4$ & $248$ & $248$&$496$ &$304$ & $144$ \\
$\{2, 5, 6\}$ & $3$ & $800$ & $520$&$480$ &$240$ &  \\
$\{2, 5, 7\}$ & $4$ & $1352$ & $1352$&$1032$ &$1104$ & $576$ \\
$\{2, 6, 7\}$ & $4$ & $5616$& $4176$&$5472$ &$7776$ & $4320$ \\
\hline
$\{3, 4, 5\}$ & $2$ & $432$ & $216$&$144$ & &  \\
$\{3, 4, 6\}$ & $3$ & $624$ & $624$&$336$ &$240$ &  \\
$\{3, 4, 7\}$ & $4$ & $852$& $852$&$1704$ &$960$ & $720$ \\
$\{3, 5, 6\}$ & $3$ & $4200$ & $1920$&$1200$ &$960$ &  \\
$\{3, 5, 7\}$ & $4$ & $6975$& $6975$&$3915$ &$2565$ & $2160$ \\
$\{3, 6, 7\}$ & $4$ & $35640$& $23400$&$22320$ &$15840$ & $14400$ \\
\hline
$\{4, 5, 6\}$ & $3$ & $14400$ & $8640$&$8640$ &$8640$ &  \\
$\{4, 5, 7\}$ & $4$ & $24000$& $24000$&$15360$ &$16320$ & $17280$ \\
$\{4, 6, 7\}$ & $4$ & $184320$& $86400$&$63360$ &$74880$ & $86400$ \\
\hline
$\{5, 6, 7\}$ & $4$ & $648000$& $432000$&$518400$ &$777600$ & $1036800$ \\
\end{longtable}
\noindent \begin{longtable}{|p{2.2cm}|p{0.5cm}|p{1.8cm}|p{1.8cm}|p{1.8cm}|p{1.8cm}|}
\hline
 \multicolumn{6}{|c|}{$|I| = 4$} \\
 \hline 
$I$ & $s$ & $C_0$ & $C_1$ & $C_2$ & $C_3$  \\
\hline
\endfirsthead
 \hline 
$I$ & $s$ & $C_0$ & $C_1$ & $C_2$ & $C_3$ \\
\hline
\endhead
\hline
\endfoot
\hline
\endlastfoot
$\{1, 2, 3, 4\}$ & $0$ & $1$ & & &  \\
$\{1, 2, 3, 5\}$ & $1$ & $1$ & $1$& &  \\
$\{1, 2, 3, 6\}$ & $2$ & $1$ &$1$ & $2$&  \\
$\{1, 2, 3, 7\}$ & $3$ & $1$ &$1$ & $2$& $6$ \\
$\{1, 2, 3, 5\}$ & $1$ & $1$ & $1$& &  \\
$\{1, 2, 4, 5\}$ & $1$ & $6$ & $2$& &  \\
$\{1, 2, 4, 6\}$ & $2$ & $7$ &$7$ & $3$&  \\
$\{1, 2, 4, 7\}$ & $3$ & $8$ &$8$ & $16$& $8$ \\
$\{1, 2, 5, 6\}$ & $2$ & $26$ &$16$ & $12$&  \\
$\{1, 2, 5, 7\}$ & $3$ & $37$ &$37$ & $27$& $24$ \\
$\{1, 2, 6, 7\}$ & $3$ & $132$ &$96$ & $120$& $144$ \\
\hline
$\{1, 3, 4, 5\}$ & $1$ & $27$ & $6$& &  \\
$\{1, 3, 4, 6\}$ & $2$ & $32$ &$32$ & $8$&  \\
$\{1, 3, 4, 7\}$ & $3$ & $37$ &$37$ & $74$& $20$ \\
$\{1, 3, 5, 6\}$ & $2$ & $183$ &$78$ & $24$&  \\
$\{1, 3, 5, 7\}$ & $3$ & $255$ &$255$ & $135$& $45$ \\
$\{1, 3, 6, 7\}$ & $3$ & $1086$ &$702$ & $636$& $240$ \\
\hline
$\{1, 4, 5, 6\}$ & $2$ & $624$ &$288$ & $144$&  \\
$\{1, 4, 5, 7\}$ & $3$ & $880$ &$880$ & $440$& $240$ \\
$\{1, 4, 6, 7\}$ & $3$ & $5792$ &$2624$ & $1536$& $960$ \\
\hline
$\{1, 5, 6, 7\}$ & $3$ & $19800$ &$11520$ & $10800$& $8640$ \\
\hline
$\{2, 3, 4, 5\}$ & $1$ & $96$ & $24$& &  \\
$\{2, 3, 4, 6\}$ & $2$ & $114$ &$114$ & $30$&  \\
$\{2, 3, 4, 7\}$ & $3$ & $132$ &$132$ & $264$& $72$ \\
$\{2, 3, 5, 6\}$ & $2$ & $800$ &$280$ & $80$&  \\
$\{2, 3, 5, 7\}$ & $3$ & $1086$ &$1086$ & $486$& $144$ \\
$\{2, 3, 6, 7\}$ & $3$ & $4752$ &$2952$ & $2304$& $720$ \\
$\{2, 4, 5, 6\}$ & $2$ & $4200$ &$1080$ & $360$&  \\
$\{2, 4, 5, 7\}$ & $3$ & $5792$ &$5792$ & $1664$& $576$ \\
$\{2, 4, 6, 7\}$ & $3$ & $39168$ &$16848$ & $5904$& $2160$ \\
$\{2, 5, 6, 7\}$ & $3$ & $159840$ &$76320$ & $43200$& $17280$ \\
\hline
$\{3, 4, 5, 6\}$ & $2$ & $14400$ &$5760$ & $2880$&  \\
$\{3, 4, 5, 7\}$ & $3$ & $19800$ &$19800$ & $8280$& $4320$ \\
$\{3, 4, 6, 7\}$ & $3$ & $159840$ &$57600$ & $25920$& $14400$ \\
$\{3, 5, 6, 7\}$ & $3$ & $972000$ &$270000$ & $140400$& $86400$ \\
\hline
$\{4, 5, 6, 7\}$ & $3$ & $3456000$ &$1728000$ & $1382400$& $1036800$ \\
\end{longtable}
\noindent \begin{longtable}{|p{2.3cm}|p{0.5cm}|p{2.5cm}|p{2.5cm}|p{2.5cm}|}
\hline
 \multicolumn{5}{|c|}{$|I| = 5$} \\
 \hline 
$I$ & $s$ & $C_0$ & $C_1$ & $C_2$  \\
\hline
\endfirsthead
 \hline 
$I$ & $s$ & $C_0$ & $C_1$ & $C_2$  \\
\hline
\endhead
\hline
\endfoot
\hline
\endlastfoot
$\{1, 2, 3, 4, 5\}$ & $0$ & $1$ & &   \\
$\{1, 2, 3, 4, 6\}$ & $1$ & $1$ & $1$&   \\
$\{1, 2, 3, 4, 7\}$ & $2$ & $1$ & $1$& $2$  \\
$\{1, 2, 3, 5, 6\}$ & $1$ & $7$ & $2$&   \\
$\{1, 2, 3, 5, 7\}$ & $2$ & $8$ & $8$&  $3$ \\
$\{1, 2, 3, 6, 7\}$ & $2$ & $30$ & $18$&  $12$ \\
$\{1, 2, 4, 5, 6\}$ & $1$ & $38$ & $6$&   \\
$\{1, 2, 4, 5, 7\}$ & $2$ & $44$ & $44$&  $8$ \\
$\{1, 2, 4, 6, 7\}$ & $2$ & $248$ & $104$&  $24$ \\
$\{1, 2, 5, 6, 7\}$ & $2$ & $852$ & $372$&  $144$ \\
\hline
$\{1, 3, 4, 5, 6\}$ & $1$ & $168$ & $24$&   \\
$\{1, 3, 4, 5, 7\}$ & $2$ & $195$ & $195$&  $30$ \\
$\{1, 3, 4, 6, 7\}$ & $2$ & $1352$ & $464$&  $80$ \\
$\{1, 3, 5, 6, 7\}$ & $2$ & $6975$ & $1710$&  $360$ \\
\hline
$\{1, 4, 5, 6, 7\}$ & $2$ & $24000$ & $8640$&  $2880$ \\
\hline
$\{2, 3, 4, 5, 6\}$ & $1$ & $600$ & $120$&   \\
$\{2, 3, 4, 5, 7\}$ & $2$ & $696$ & $696$&  $144$ \\
$\{2, 3, 4, 6, 7\}$ & $2$ & $5616$ & $1656$&  $360$ \\
$\{2, 3, 5, 6, 7\}$ & $2$ & $35640$ & $6120$&  $1440$ \\
$\{2, 4, 5, 6, 7\}$ & $2$ & $184320$ & $31680$&  $8640$ \\
\hline
$\{3, 4, 5, 6, 7\}$ & $2$ & $648000$ & $216000$&  $86400$ \\
\end{longtable}
\noindent \begin{longtable}{|p{3.2cm}|p{0.5cm}|p{3.5cm}|p{3.5cm}|}
\hline
 \multicolumn{4}{|c|}{$|I| = 6$} \\
 \hline 
$I$ & $s$ & $C_0$ & $C_1$ \\
\hline
\endfirsthead
 \hline 
$I$ & $s$ & $C_0$ & $C_1$ \\
\hline
\endhead
\hline
\endfoot
\hline
\endlastfoot
$\{1, 2, 3, 4, 5, 6\}$ & $0$ & $1$ &   \\
$\{1, 2, 3, 4, 5, 7\}$ & $1$ & $1$ & $1$  \\
$\{1, 2, 3, 4, 6, 7\}$ & $1$ & $8$ & $2$  \\
$\{1, 2, 3, 5, 6, 7\}$ & $1$ & $51$ & $6$  \\
$\{1, 2, 4, 5, 6, 7\}$ & $1$ & $272$ & $24$  \\
$\{1, 3, 4, 5, 6, 7\}$ & $1$ & $1200$ & $120$  \\
$\{2, 3, 4, 5, 6, 7\}$ & $1$ & $4320$ & $720$  \\
\end{longtable}
\noindent \begin{longtable}{|p{4.6cm}|p{0.5cm}|p{6cm}|}
\hline
 \multicolumn{3}{|c|}{$|I| = 7$} \\
 \hline 
$I$ & $s$ & $C_0$ \\
\hline
\endfirsthead
 \hline 
$I$ & $s$ & $C_0$  \\
\hline
\endhead
\hline
\endfoot
\hline
\endlastfoot
$\{1, 2, 3, 4, 5, 6, 7\}$ & $0$ & $1$  \\
\end{longtable}

\printbibliography
\end{document}